\UseRawInputEncoding

\documentclass[a4paper,reqno]{amsart}

\usepackage{graphicx}

\usepackage{hyperref}

\usepackage{amssymb,amsxtra,amsfonts,dsfont}

\usepackage{amsmath}

\usepackage[svgnames]{xcolor}

\usepackage[color,matrix,arrow]{xy}

\input xy \xyoption{all}

\newtheorem{lem}{Lemma}[section]

\newtheorem{thm}[lem]{Theorem}

\newtheorem{pro}[lem]{Proposition}

\newtheorem{cor}[lem]{Corollary}

\newtheorem{cdns}[lem]{Conditions}

\newcommand{\slrw}[1]{\stackrel{#1}{\longrightarrow}}

\newcommand{\ttn}[1]{\tau_n^{#1}}

\newcommand{\lrw}{\longrightarrow}

\newcommand{\LL}{\Lambda}

\newcommand{\xa}{\alpha}
\newcommand{\xb}{\beta}

\newcommand{\caG}{\mathcal G}
\newcommand{\bcaG}{\overline{\mathcal G}}
\newcommand{\bcaL}{\overline{\mathcal L}}
\newcommand{\caD}{\mathcal D}

\newcommand{\caM}{\mathcal M}
\newcommand{\caI}{\mathcal I}

\newcommand{\caU}{\mathcal U}
\newcommand{\caP}{\mathcal P}
\newcommand{\caS}{\mathcal S}

\newcommand{\RHom}{\mathbf{R}\strut\kern-.2em\operatorname{Hom}\nolimits}

\newcommand{\zZ}{{\mathbb Z}}

\newcommand{\tL}{\widetilde{\LL}}

\newcommand{\dtL}{\Delta{\LL}}
\newcommand{\dtsL}{\Delta_{\sigma}{\LL}}
\newcommand{\dtnL}{\Delta_{\nu}{\LL}}

\newcommand{\Db}{\mathcal D^b}

\newcommand{\olL}{\overline{\LL}}

\newcommand{\GG}{\Gamma}

\newcommand{\trho}{\tilde{\rho}}

\newcommand{\olG}{\overline{\Gamma}}

\newcommand{\tQ}{\widetilde{Q}}

\newcommand{\olQ}{\overline{Q}}

\newcommand{\olrho}{\overline{\rho}}

\newcommand{\OO}{\Omega}

\newcommand{\identity}{\mathds{1}}
\newcommand{\hhom}{\mathrm{Hom}}
\newcommand{\hhm}{\mathrm{Hom}}

\newcommand{\Hom}{\mathrm{Hom}}

\newcommand{\End}{\mathrm{End}}
\newcommand{\Ext}{\mathrm{Ext}}
\newcommand{\Ho}{{\rm H}}

\newcommand{\zzs}[1]{{\mathbb Z|_{#1} }}

\newcommand{\uhom}{\underline{\mathrm{Hom}}}

\newcommand{\ind}{\mathrm{ind}\,}

\newcommand{\Ker}{\mathrm{Ker}\,}

\newcommand{\add}{\mathrm{add}\,}

\newcommand{\tr}{\mathrm{Tr}\,}

\newcommand{\gl}{\mathrm{gl.dim}\,}

\newcommand{\mmod}{\mathrm{mod}\,}

\newcommand{\bfP}{\mathbf{P}}

\newcommand{\caC}{\mathcal{C}}

\newcommand{\caA}{{\mathcal A}}

\newcommand{\dmn}{\mathrm{dim}\,}

\newcommand{\arr}[2]{\begin{array}{#1}#2\end{array}}

\newcommand{\eqqc}[2]{\begin{equation}\label{#1}#2\end{equation}}

\newcommand{\eqqcn}[2]{\[ #2 \]}
\newcommand{\eqqcnn}[2]{$ #2 $}

\begin{document}
\title[$\mathbb Z Q$ type constructions]{$\mathbb Z Q$ type constructions in higher representation theory}
\author[Guo, Lu and Luo]{Jin Yun Guo and Xiaojian Lu\\ LCSM( Ministry of Education), School of Mathematics and Statistics\\ Hunan Normal University\\ Changsha, Hunan 410081, P. R. China\\
Deren Luo \\School of Mathematics \\ Hunan Institute of Science and Technology\\Yueyang, Hunan 414000, P. R. China}
\email{gjy@hunnu.edu.cn(Guo), 774270313@qq.com(Lu), luoderen@126.com(Luo)}

\thanks{This work was supported by National Natural Science Foundation of China  under Grants \#11271119 and \#11671126 and the Construct Program of the Key Discipline in Hunan Province}

\maketitle

\begin{abstract}
Let  $Q$ be an acyclic quiver, it is classical that certain truncations of the translation quiver $\mathbb Z Q$ appear in the Auslander-Reiten quiver of the path algebra $kQ$.
The stable $n$-translation quiver $\mathbb Z|_{n-1} Q$ is introduced  as a generalization of the $\mathbb Z Q$ construction in studying higher representation theory of algebras for an acyclic bound quiver $Q$.
In this paper, we find conditions for a Hom-finite Krull-Schmidt $k$-category to be realized as the bound path category of a convex full subquiver of an stable $n$-translation quiver.We show that for  $n$-slice algebra $\Gamma$, which is an $n$-hereditary algebra whose $(n+1)$-preprojective algebra is $(q+1,n+1)$-Koszul,  with bound quiver $Q^{op}$, its $n$-preprojective and $n$-preinjective components in the module category and truncations of  the stable $n$-translation quiver $\mathbb Z|_{n-1} Q^{op}$. 
We also use $\mathbb Z|_{n-1} Q^{op}$ to describe the $\nu_n$-closure of $\Gamma$ in the derived category.


\end{abstract}

%

\section{Introduction}
\label{intro}

The path algebra of an acyclic quiver $Q$ has very nice representation theory, and the translation quiver $\zZ Q$ has played an important role in its representation theory.
The preprojective component and preinjective component of the Auslander-Reiten quiver of the path algebra $kQ$ are realized as truncations of the translation quiver $\zZ Q$ \cite{ars95,r84}.
The translation quiver $\zZ Q$ also appears in the derived category $\caD^b(kQ)$  as components arising from the preprojective and preinjective  modules \cite{h88}.
On the other hand, the quiver $Q$ is a slice of $\zZ Q$ and $kQ$ is the algebra defined by a slice of $\zZ Q$.
In this paper, we study a class of algebras for which the $n$-translation quiver $\zZ|_{n-1} Q$, a generalized construction of $\zZ Q$ (see Section \ref{sec:qzq}), plays a similar role in the setting of higher representation theory.

Higher representation theory is developed by Iyama and his coauthors \cite{hio14,i007,i07,i11,io11}, and is widely used in representation theory of algebra and noncommutative algebraic geometry \cite{mm11,himo}.
We observed that graded self-injective algebras bear certain feature of higher representation theory \cite{gw00,g12}, and we introduce $n$-translation algebras for studying higher representation theory related to graded self-injective algebras \cite{g16}.

We introduce $n$-slice and $n$-slice algebra related to a stable $n$-translation quiver in Section \ref{sec:qzq}.
Such quiver and algebra are quadratic dual of the complete $\tau$-slice and the $\tau$-slice algebra of a stable $n$-translation quiver.
In fact, $n$-slice algebras are exactly $n$-hereditary algebras whose $(n+1)$-preprojective algebras are $(n+1,q+1)$-Koszul (see Theorem \ref{suf_nec}).
Let $\GG$ be an $n$-slice algebra with bound quiver $Q^{\perp}$ and let $\LL=\GG^{!,op}$ be its quadratic dual with bound quiver $Q$.
Then $\zzs{n-1}Q$ is an $n$-translation quiver, and it is the bound quiver of the $n$-translation algebra $\olL=\dtL\# k\zZ^*$, the smash product of the trivial extension $\dtL$ of $\LL$ with $\zZ^*$ (Proposition \ref{0nap:extendible1}).
In this case, $Q$ is a complete $\tau$-slice in $\zzs{n-1}Q$ and $\LL$  a $\tau$-slice algebra of $\zzs{n-1}Q$.
Using results in \cite{g16}, certain full convex subquiver of $\zzs{n-1}Q$ have the property that the bound path category of its quadratic dual  has $n$-almost split sequences (see Section \ref{sec:pre}).
In this paper, we describe when a Hom-finite Krull-Schmidt $k$-category is equivalent to the bound path category of a convex full subquiver of  $\zzs{n-1}Q^{\perp}$.
Using this equivalence, we show that 
the $n$-preinjective component and $n$-preprojective component of $n$-slice algebra  are realized  as the bound path categories of some  truncations of $\mathbb Z|_{n-1} Q^{op,\perp}$, and their Auslander-Reiten quivers are truncations of $\mathbb Z|_{n-1} Q^{op,\perp}$.
We also show that for an $n$-slice algebra $\GG$  of infinite type, the $\nu_n$-closure of $\GG$ in the derived category has the Auslander-Reiten quiver $\mathbb Z|_{n-1} Q^{op,\perp}$.

The paper is organized as follows.
In Section \ref{sec:pre}, concepts and results needed in this paper are recalled.
In Section \ref{sec:qzq},  $n$-properly-graded quivers, $n$-slice quivers and $n$-slice algebras are introduce, and the $\zZ|_{n-1} Q$ construction for these quivers are recalled.
In Section \ref{sec:bpc}, the bound quiver categories, especially, those arise from the quadratic dual of a stable $n$-translation quivers are discussed.
The category of finite generated projective modules of the quadratic dual of an acyclic $n$-translation algebra is also dealt in this section, and conditions for a Hom-finite Krull-Schmidt $k$-category to be realized as the  bound path category of a truncation of an $n$-translation quiver $\zzs{n-1} Q^{op,\perp}$ are found.
It is proven in Section \ref{sec:ttn} that for an $n$-slice algebra with bound quiver $Q^{\perp}$, its  $n$-preprojective and $n$-preinjective components are truncations of $\mathbb Z|_n Q^{\perp}$ and in Section \ref{sec:nu} that when an $n$-slice algebra is of infinite type and satisfies Condition \ref{conditionbb}, their $\nu_n$-closure in the derived category is equivalent to the bound path category of $\mathbb Z|_n Q^{op,\perp}$.

\section{Preliminary}
\label{sec:pre}

Let $k$ be a field, and let $\LL = \LL_0 + \LL_1+\cdots$ be a graded algebra over $k$ with $\LL_0$ a direct sum of copies of $k$ such that $\LL$ is generated by $\LL_0$ and $\LL_1$.
Such algebra is determined by a bound quiver $Q= (Q_0,Q_1, \rho)$ \cite{g16}.

Recall that a bound quiver $Q= (Q_0,Q_1, \rho)$ is a quiver with $Q_0$ the set of vertices, $Q_1$ the set of arrows and $\rho$ a set of relations.
In this paper, the vertex set $Q_0$ may be infinite and the arrow set $Q_1$ is assumed to be locally finite.
The arrow set $Q_1$ is usually defined with two maps $s, t$ from $Q_1$ to $Q_0$ to assign an arrow $\alpha$ its starting vertex $s(\alpha)$ and its terminating vertex $t(\alpha)$.
Write $Q_t$ for the set of paths of length $t$ in the quiver $Q$, and write $kQ_t$ for space spanned by $Q_t$.
$kQ =\bigoplus_{t\ge 0} kQ_t$ is the path algebra of quiver $Q$ over $k$.
We also write  $s(p)$ for the starting vertex of a path $p$ and $t(p)$ for the terminating vertex of $p$.
The relation set $\rho$ is a set of linear combinations of paths of length $\ge 2$.
We assume that it is normalized such that each element in $\rho$ is a linear combination of paths starting at the same vertex and ending at the same vertex, we also assume that these paths are of the same length since we deal with graded algebra.
Conventionally, modules are assumed to be left module in this paper.

Let $\LL_0=  \bigoplus\limits_{i\in Q_0} k_i$, with $k_i \simeq k$ as algebras, and let $e_i$ be the image of the identity in the $i$th copy of $k$ under the canonical embedding.
Then $\{e_i| i \in Q_0\}$ is a complete set of orthogonal primitive idempotents in $\LL_0$ and $\LL_1 = \bigoplus\limits_{i,j \in Q_0 }e_i \LL_1 e_j$ as $\LL_0 $-bimodules.
Fix a basis $Q_1^{ij}$ of $e_i \LL_1 e_j$ for each pair $i, j\in Q_0$, take the elements of $Q_1^{ij}$ as arrows from $i$ to $j$, and let $Q_1= \cup_{(i,j)\in Q_0\times Q_0} Q_1^{ ij}.$
Then $\LL \simeq kQ/(\rho)$ as algebras for some relation set $\rho$, and $\LL_t \simeq kQ_t /((\rho)\cap kQ_t) $ as $\LL_0$-bimodules.

The {\em opposite (bound) quiver} $Q^{op} = (Q_0^{op},Q_1^{op},\rho^{op})$ of a bound quiver $Q = (Q_0,Q_1,\rho)$ is obtained from $Q$ by reversing all the arrows.
So $Q_0^{op}= Q_0$, $Q_1^{op} = \{\xa^{op}: j \to i|\xa: i\to j \in Q_1\}$.
If $p=\xa_m\cdots \xa_1$ is a path in $Q$, write $p^{op} = \xa_1^{op}\cdots\xa_m^{op}$ as a path in $Q^{op}$.
Thus $\rho^{op} = \{ \sum_{p_t \in Q_r} a_{t} p^{op}_t| \sum_{p_t \in Q_r} a_{t} p_t\in \rho\}$.
Recall that for an algebra $\LL$, its {\em opposite algebra} $\LL^{op}$ is the algebra defined on $\LL$ with the multiplication defined by $a\circ b = ba$ for any $a,b\in \LL $.
Clearly if $\LL = kQ/ (\rho)$, then $\LL^{op} = kQ^{op}/ (\rho^{op})$, so the opposite algebra is exactly the algebra defined by the opposite bound quiver.

We are mainly interested in  quadratic bound quiver,  that is, $\rho$ can be chosen a set of linear combination of paths of length $2$.
In this case, $\LL$ is called a {\em quadratic algebra}.
For a quadratic bound quiver $Q= (Q_0,Q_1,\rho)$, we identify $Q_0$ and $Q_1$ with their dual bases in the local dual spaces $\bigoplus_{i \in Q_0} D k e_i$ and $\bigoplus_{i,j \in Q_0} De_j kQ_1 e_i$, respectively.
That is, we regard the idempotents $e_i$ as the linear functions on $\bigoplus_{i\in Q_0} k e_i$ such that $e_i(e_j) = \delta_{i,j}$, and for $i,j\in Q_1$, and arrows from $i$ to $j$ are also regarded as the linear functions on $\bigoplus_{i,j\in Q_1}e_j k Q_1 e_i$ such that for any arrows $\xa, \xb$, we have $\xa(\xb) =\delta_{\xa,\xb}$.
By defining $\xa_1 \cdots \xa_r (\xb_1, \cdots, \xb_r)=\xa_1(\xb_1) \cdots \xa_r ( \xb_r) $, $e_j kQ_r e_i$ is identified with its dual space for each $r$ and each pair $i,j$ of vertices, and the set of paths of length $r$ is the dual basis to itself in $e_j kQ_r e_i$.
Take a spanning set $ \rho^{\perp}_{i,j}$ of orthogonal subspace of the set $e_j k\rho e_i$ in the space $e_j kQ_2 e_i$ for each pair $i,j \in Q_0$, and set \eqqc{relationqd}{\rho^{\perp} = \cup_{i,j\in Q_0} \rho^{\perp}_{i,j}.}
The {\em quadratic dual quiver} of $Q$ is defined as the bound quiver $Q^{\perp} =(Q_0,Q_1, \rho^{\perp})$, and the {\em quadratic dual algebra} of $\LL$ is the algebra $\LL^{!,op} \simeq kQ/(\rho^{\perp})$ defined by the bound quiver $Q^{\perp}$ (see \cite{g20}).

\medskip

The $n$-translation quiver and $n$-translation algebra are introduced in \cite{g16}.
The bound quiver of a graded self-injective algebra of Loewy length $n+2$ is exactly  the  stable $n$-translation quiver with the $n$-translation $\tau$ induced by the Nakayama permutation.
Here stable means  that the $n$-translation $\tau$ is defined for all the vertices, which sends each vertex $i$ to the starting vertex $\tau i$ of a bound path of length $n+1$ ending at $i$.
Stable $n$-translation quiver is called stable bound quiver of Loewy length $n+2$ in \cite{g12}.

$\tau$-hammock is introduced in \cite{g12} for a stable $n$-translation quiver.
Let $\olQ= (\olQ_0,\olQ_1, \overline{\rho} )$ be an $n$-translation quiver with $n$-translation $\tau$ and let $\olL = k\olQ/(\olrho)$.
For each non-injective vertex  $i\in \olQ_0$,  $\tau^{-1} i $ is defined in $\olQ$,  the {\em $\tau$-hammock $H_i$ ending at $i$} is defined  as the quiver with the vertex set \eqqcn{hammockvl}{H_{i,0}=\{ (j,t) | j \in \olQ_0, \exists p\in \olQ_t, s(p)=j, t(p)=i,  0\neq p\in \olL \},} the arrow set \eqqcn{hammockal}{\arr{ll}{ H_{i,1} = & \{ (\alpha , t): (j,t+1) \longrightarrow (j', t)|\alpha: j\to j'\in \olQ_1, \\ &\quad  \exists p\in \olQ_t, s(p)=j', t(p)=i, 0\neq p\xa \in \olL_{t+1} \},}} and a hammock function $\mu_i: H_{i,0} \longrightarrow \zZ$ which is integral map on the vertices defined by \eqqcn{hammockfl}{\mu_i (j,t)=\dmn_k e_j \olL_t e_i} for $(j,t)\in H_{i,0}$.

Dually, for non-projective vertex $i$,  $\tau i$ is defined in $\olQ$, the  {\em $\tau$-hammock $H^i$ starting at $i$} is as the quiver with the vertex set \eqqcn{hammockvr}{H^i_0=\{ (j,t) | j \in \olQ_0, \exists p\in \olQ_t,  s(p)=i, t(p)=j, 0\neq p\in \olL \}, } the arrow set \eqqcn{hammockar}{ \arr{ll}{H^i_1= &\{ (\alpha , t): (j,t) \longrightarrow (j', t+1)|\alpha: j\to j'\in \olQ_1, \\ &\quad \exists p\in \olQ_t, s(p)=i, t(p)=j, 0\neq \alpha p\in \olL_{t+1}  \},}} and a hammock functions $\mu^i: H^i_0 \longrightarrow \zZ$ which is the integral maps on the vertices  defined by \eqqcn{hammockfr}{ \mu^i (j,t)=\dmn_k e_i \olL_t e_j} for $(j,t)\in H^i_0$.
It is easy to see that $ (j,t) \to (j, n+1-t) $ defines a bijective map from $H_{i,0}$ to $H^{\tau i}_0$ preserving the arrows.

Observe that when $n=1$, the $1$-translation quiver is the usual translation quiver, and a $\tau$-hammock is a mesh in the quiver.
We may regard the $\tau$-hammock as a generalization of the mesh in the translation quiver.

If the projection to the first component of the vertices is a monomorphism, the hammocks $H^i$ and $H_i$ are regarded as  subquivers of $\olQ$.
In this case, $H_i = H^{\tau i}$ when $\tau i$ is defined and $H^{i} = H_{\tau^{-1} i }$ when $\tau^{-1} i$  is defined.
The $\tau$-hammock $H_i = H^{\tau^{-1} i}$ describes the support of the projective cover of the simple $\olL$-module corresponding to the vertex $\tau^{-1} i$ and the injective envelop of the simple $\olL$-module corresponding to the vertex $i$.

\medskip
Recall that an algebra defined by an $n$-translation quiver is called {\em an  $n$-translation algebra} if it is an $(n+1,q)$-Koszul algebra for some $q\ge 2$ or $q=\infty$ \cite{g16}.
An $n$-translation algebra is quadratic \cite{bbk02}.

Assume that  $\olL$ is an $n$-translation algebra which is $(n+1,q)$-Koszul, and let  $\olG = \olL^{!,op}$ be its quadratic dual.
Let $\bcaG = \add(\olG)$ be the category of finite generated projective $\olG$-modules.
The $\tau$-hammocks are used to describe the  Koszul complexes for $\olG$ in \cite{g16}.
Similar to Proposition 7.4 of \cite{g16}, we have the following characterization of the Koszul complexes in $\bcaG$.

\begin{pro}\label{KoszulinG} For each vertex $i\in \olQ_0$, we have a Koszul complex
\eqqc{nassinsubG}{\xi_i: M_{n+1}=\olG e_{\tau^{-1} i } \slrw{f} \cdots \lrw M_t \lrw \cdots \slrw{\phi} M_0= \olG e_i } with $M_t = \bigoplus_{(j,t)\in H^{i}_{0}}  (\Gamma e_j)^{\mu^i(j,t)} $ for $0\le t\le n$ and a Koszul complex
\eqqc{nassinsubG1}{\zeta_i: M_{n+1}= \olG e_{ i } \slrw{f} \cdots \lrw M_t \lrw \cdots \slrw{\phi} M_0=\olG e_{\tau^{} i} } with $M_t = \bigoplus_{(j,n+1-t)\in H_{i,0}}  (\Gamma e_j)^{\mu_i(j,n+1-t)} $ for $0\le t\le n$.
\end{pro}

We see that $\zeta_i =\xi_{\tau^{} i}$ and $\xi_i =\zeta_{\tau^{-1} i}$.

\medskip

Iyama recently introduces higher almost split sequences, now we recall the related notions \cite{i11}.
Let $\caC$ be a Krull-Schmidt category, and let $J_{\caC}$ be its Jacobson radical.

(a). We call a complex \eqqc{Eqqa}{ M\stackrel{f_0}{ \longrightarrow} C_1 \stackrel{f_1}{\longrightarrow} C_2 \stackrel{f_2}{\longrightarrow} \cdots } a {\em source sequence} of $M$ if the following conditions are satisfied.
(i).  $C_i \in \caC$ and $f_i\in J_{\caC}$ for any $i$,
(ii).  we have the following exact sequence on $\caC$, \eqqc{Eqqb}{ \cdots \stackrel{f_2}{\longrightarrow} \Hom_{\caC}(C_2,-) \stackrel{f_1}{\longrightarrow} \Hom_{\caC}(C_1,-) \stackrel{f_0}{\longrightarrow} J_{\caC}(M,-) \longrightarrow 0 .}

A {\em sink sequence} are defined dually.

(b). We call a complex \eqqcn{}{M \stackrel{}{\longrightarrow} C_1 \stackrel{}{\longrightarrow} C_2 \stackrel{}{\longrightarrow} \cdots \stackrel{}{\longrightarrow} C_{n} \stackrel{}{\longrightarrow} N } an {\em $n$-almost split sequence} if this is a source sequence of $M\in\caC$ and a sink sequence of $N\in\caC$.

\medskip

By Theorem 7.2 of \cite{g16}, the Koszul complexes \eqref{nassinsubG} is an $n$-almost split sequence in $\bcaG$ if $q$ is infinite or $q$ is finite with $\olG_q e_{\tau^{-1} i}=0$  for any $i\in \olQ_0 \setminus \mathcal I$.

Now we consider the $n$-almost split sequence in a subcategory $\bcaG'$ of $\bcaG$.

\begin{lem}\label{nass:objinsubcat}
Let $\bcaG'$ be a full subcategory of $\bcaG$ such that objects appearing in $\xi_i$ of $\eqref{nassinsubG}$ are all in $\caG'$.
Then $\xi_i$ of $\eqref{nassinsubG}$ is an $n$-almost split sequence in $\bcaG'$ if and only if $q$ is infinite or $q$ is finite and   $\hhom_{\bcaG}(X,\olG e_{\tau^{-1}i} )_q=0$ for any $X \in \bcaG'$.
\end{lem}
\begin{proof}
By Lemma 7.1 in \cite{g16}, $\xi_i$ of $\eqref{nassinsubG}$ is an $n$-almost split sequence if and only if it is exact at $\olG e_{\tau^{-1}i}$.
Since $\olL$ is $n$-translation algebra, so it is $(n+1,q)$-Koszul algebra, and $\olG $ is $(q,n+1)$-Koszul algebra.
Thus $\ker f$ in $\eqref{nassinsubG}$ is exactly $\olG_q e_{\tau^{-1}i}$.
So $\hhom_{\bcaG}(\olG e_j,\olG e_{\tau^{-1}i} )_{q}=0$ for all $j\in \olQ_0$, if and only if $\olG_q e_{\tau^{-1}i}=0$ and $\xi_i$ of $\eqref{nassinsubG}$ is exact at $\olG e_{\tau^{-1}i}$.
\end{proof}

We can restate Lemma \ref{nass:objinsubcat} using bound paths.
\begin{cor}\label{nass:subcat}
Under the assumption of Lemma \ref{nass:objinsubcat}, then $\xi_i$ of $\eqref{nassinsubG}$ is an $n$-almost split sequence in $\bcaG'$ if and only if  $q$ is infinite or $q$ is finite and for any $j\in \olQ_0$, $\olG e_{j}$ is not an object of $\bcaG'$ whenever there is a bound path of length $q$ from $\tau^{-1} i$ to $j$ in $\olQ$.
\end{cor}

We call a convex full subquiver $\olQ'$ of a bound quiver $\olQ$ a truncation of $\olQ$.
Obviously, a truncation $\olQ'$ of an $n$-translation quiver is also an $n$-translation quiver.
A truncation $\olQ'$ of $\olQ$ is called {\em $\tau$-mature} if  it satisfies the condition in Corollary \ref{nass:subcat}.
As a consequence of Theorem 7.2 of \cite{g16}, we have the following Proposition.

\begin{pro}\label{taumat}
If $\olQ'$ is a $\tau$-mature truncation  of $\olQ$.
Let $\olL'$ be the algebra defined by $\olQ'$ and let $\olG'= \olL^{!,op}$ be its quadratic dual.
Then $\add \olG'$ has $n$-almost split sequences, that is, if both $i$ and $\tau^{-1} i$ are in $\olQ'$, then $\xi_i$ of  \eqref{nassinsubG}  is an $n$-almost split sequence in $\add \olG'$.
\end{pro}

So if $\olL$ is Koszul, then $q=\infty$, and  every convex subquiver $\olQ'$ of $\olQ$ is $\tau$-mature.
If $q < \infty$, then $\olG$ is a $(q, n+1)$-Koszul, thus it is a $(q-1)$-translation quiver with $(q-1)$-translation $\tau_{\perp}$.
Any maximal bound path in $\olQ^{\perp}$ has length $q$ and are from $\tau_{\perp} i$ to $i$.
In this case, $\olQ'$ is $\tau$-mature if for any $i\in \olQ'_0$, either $\tau i$ or $\tau^{-1}_{\perp} i$ is not in $\olQ'_0$

\medskip

We say that a subcategory $\bcaG'$ has {\em $n$-almost split sequences} if $\xi_i$ of $\eqref{nassinsubG}$ is an $n$-almost split sequence in $\bcaG'$ whenever $\olG e_i$ and $\olG e_{\tau^{-1} i}$ are both in $\bcaG'$.
If $\olQ'$  is a full subquiver of $\olQ$, let $\bcaG'= \add\{\olG e_j| j\in Q_0'\}$ be a full subcategory of $\bcaG$, and call it the {\em subcategory defined by $\olQ'$}.
If $\olQ'$ is convex, then  objects in $\eqref{nassinsubG}$ are all in $\bcaG'$ whenever $i$ and $\tau^{-1} i$ are both in $\olQ'_0$.
By Corollary \ref{nass:subcat}, we have the following proposition.

\begin{pro}\label{GGnass}
Let $\olG'$ be the quadratic dual algebra of a bound quiver algebra $\olL'$ defined by a $\tau$-mature subquiver $\olQ'$ of the stable $n$-translation quiver $\olQ$, then $\add ( \olG')$ has $n$-almost split sequences.
\end{pro}

\medskip

Given a bound quiver $\olQ=(\olQ_0,\olQ_1, \rho)$, we can relate to it the path category and bound path category.
The {\em path category $\bfP(\olQ)$} of the quiver  $\olQ=(\olQ_0,\olQ_1)$ is defined as  the $k$-additive category with objects $\add (\olQ_0)$, and the morphism space $\Hom_{\bfP(\olQ)}(i,j)$ the vector space with the basis the paths from $i$ to $j$ for $i,j \in Q_0$.
So the indecomposable objects are the vertices and we have \eqqcn{homspace}{\Hom_{\bfP(\olQ)}(i,j) = e_j k\olQ e_i  = e_i k\olQ^{op} e_j =\Hom_{k\olQ^{op}}(k\olQ^{op}e_i,k\olQ^{op} e_j) } as vector spaces.
If $\olQ$ is locally finite and acyclic, the path category of $\olQ$ coincides with its complete path category.
The {\em bound path category} $\caP(\olQ) = \bfP(\olQ)/\caI$ of the bound quiver $\olQ=(\olQ_0,\olQ_1, \rho)$  is the quotient category of the path category modulo the ideal $\caI$ generated by the relations in $\rho$ (see  Section 6 of \cite{i11}).
Let $\olL$ be the algebra defined by the bound quiver $\olQ$, then $\olL^{op}$ is the algebra defined by the bound quiver  $\olQ^{op}$.
Thus the objects of $\caP(\olQ) $ are the same as those of $\bfP(\olQ)$, $i \to \olL^{op} e_i$ define a bijective correspondence from the objects in $\caP(\olQ) $ and the objects in $\add \olL^{op}$  and we have \eqqcn{homcatalg}{ \Hom_{\caP(\olQ)} (i, j) = e_i k\olQ e_j/ e_i (\rho) e_j \simeq \Hom_{\olL^{op}}(\olL e_i, \olL e_j)} as vector spaces.
So we have the following consequence.

\begin{pro}\label{pathalgcat}
Assume that $\olQ=(\olQ_0,\olQ_1,\rho)$ is an acyclic locally finite bound quiver.
Let $\olL$ and $\caP(\olQ)$ be the bound path algebra and bound path category of $\olQ$, respectively.
Then we have an equivalence of categories $ \caP(\olQ)\simeq\add (\olL^{op}) $, where $\add (\olL^{op})$ is the category of finitely generated left projective $\olL^{op}$-modules.
\end{pro}

Clearly,  $\add (\olL)$ is equivalent to $\caP(\olQ^{op})$.

\medskip

Assume that $\caC$ is a Hom-finite Krull-Schmidt category with $n$-almost split sequences.
Similar to Definition 6.1 of \cite{i11}, one defines {\em the Auslander-Reiten quiver} of $\caC$  as the quiver with the vertices the iso-classes of indecomposable objects and the number of arrows from $x$ to $y$ is \eqqc{noarrows}{d_{xy} = \dmn_k J_{\caC}(x,y) / J^2_{\caC}(x,y), } with the translation induced by the $n$-almost split sequences.

In fact, the arrows are determined by source morphisms and sink morphisms (see Lemma 6.1 of \cite{i11} and Proposition 3.1 of \cite{birsm11}).

\begin{lem}\label{category:arrows}

If $M \to Y$ is a sink map in $\caC$, then $d_{X,Y}$ is equal to the number of $X$ appearing in the direct sum decomposition of $M$.

If $X \to M$ is a source map in $\caC$, then $d_{X,Y}$ is equal to the number of $Y$ appearing in the direct sum decomposition of $M$.
\end{lem}

Assume $\phi: \bfP(\olQ) \to \caC$  assigns a vertex $i$ in $\olQ_0$ to an indecomposable object $\phi(i)$ in $\caC$, and for each pair $i,j$ in $\olQ_0$, assigns a morphism $\phi(a) \in J_{\caC} (\phi(i), \phi(j) )$ for each arrow $a : i \to j$ in $\olQ$, and that the image of the set $\{\phi(a) | s(a) = i, t(a) = j \}$ is a $k$-basis of $J_{\caC} (\phi(i), \phi(j)) /J^2_{\caC} (\phi(i), \phi(j) )$ for any $i, j \in \olQ_0$.
Then $\phi$ extends uniquely to a full dense functor \eqqcn{func}{\phi : \bfP(\olQ) \to  \caC.}
If $\olQ$ is acyclic, we have that \eqqcn{completealg}{e_i k\olQ e_j = e_i \widehat{k\olQ} e_j}  for any vertices $i, j \in \olQ_0$, for the $(\olQ_1)$-adic completion $\widehat{k\olQ}$ of $k\olQ$.
Thus \eqqcn{complete}{e_i k\olQ/(\rho) e_j = e_i \widehat{k\olQ}/\overline{(\rho)} e_j.}
Similar to Lemma 6.4 of \cite{i11} and Proposition 3.6 of \cite{birsm11}, we can determine the relation set, that is, a set of the generators of $\Ker \phi$, using $n$-almost split sequences.
In fact, the relations are determined by the first two maps in the source sequences or by the last two maps in the sink sequences.

\begin{lem}\label{category:relation}
Let $\caC$  be a Hom-finite Krull-Schmidt $k$-category.
Let $\olQ$ be an acyclic locally finite quiver, and let $\bfP(\olQ)$ be the path category of $\olQ$ over $k$.
Assume that we have a fully dense functor $\phi: \bfP(\olQ) \to \caC$.
\begin{enumerate}
\item If any $i\in \olQ_0$ has the source sequence with the first two terms
    \eqqc{soueceseq}{
    \phi(i)  \slrw{(\alpha)_{\alpha\in \olQ_1,s(\alpha)=i}} \bigoplus_{\alpha\in \olQ_1,s(\alpha)=x} \phi(t(\alpha)) \slrw{(\gamma_{\alpha,h})_{1\le h \le m_i}} \bigoplus_{1\le h\le m_i} \phi(t(\gamma_{\alpha,h})),
    }
    then the kernel of $\phi$ is generated by $$\{ \sum_{ \alpha\in \olQ_1, s(\alpha)=i} \gamma_{\alpha,h} \alpha | i\in \olQ_0, 1 \le  h \le  m_i\}.$$

\item If any $j\in \olQ_0$ has the sink sequence with the last two terms
    \eqqc{sinkseq}{
    \bigoplus_{1\le j\le m'_j} \phi(s(\gamma'_{h,\alpha})) \slrw{ (\gamma'_{h,\alpha} )_{ 1\le h \le m'_j}} \bigoplus_{\alpha\in \olQ_1,t(\alpha)=j} \phi(s(\alpha))  \slrw{(\alpha)_{\alpha\in \olQ_1,t(\alpha)=j}} \phi(j),
    }
    then the kernel of $\phi$ is generated by $$\{ \sum_{ \alpha\in \olQ_1, t(\alpha)=j} \alpha \gamma'_{h,\alpha}  | j\in \olQ_0, 1 \le  h \le  m'_j\}.$$
\end{enumerate}
\end{lem}

\section{$n$-slice quiver $Q$ and the $n$-translation quiver $\zzs{n-1} Q$}
\label{sec:qzq}
A bound quiver $Q =(Q_0,Q_1,\rho)$  is called {\em $n$-properly-graded} if all the maximal bound paths have the same length $n$.
The (graded) algebra defined by an  $n$-properly-graded quiver is called an {\em $n$-properly-graded algebra.}
The quadratic dual quiver $Q^{\perp}$ of an  acyclic quadratic $n$-properly-graded quiver $Q$ is called an {\em $n$-slice quiver} or an $n$-slice.

Obviously, we have the following Proposition.

\begin{pro}\label{oppn}
If $Q$ is an $n$-properly-graded quiver, so is $Q^{op}$.

If $Q^{\perp}$ is an $n$-slice, so is $Q^{op, \perp}$.

If $\LL$ an $n$-properly-graded algebra, so is $\LL^{op}$.
\end{pro}

For a finite acyclic $n$-properly-graded quiver  $Q$, we can relate an infinite bound  quiver $\zZ|_{n-1} Q$ to it as in \cite{g16}.

Let $\LL$ be the algebra defined by the $n$-properly-graded quiver $Q $.
Let $\caM$ be a set of maximal linearly independent set of maximal bound paths in $Q$.
We define the {\em returning arrow quiver } $\tQ= (\tQ_0,\tQ_1, \trho)$ of $Q$ as the bound quiver with the vertex set $\tQ_0 = Q_0$, the arrow set $\tQ_1 = Q_1\cup Q_{1,\caM}$, here \eqqc{retarr}{Q_{1,\caM} = \{\beta_{p}: t(p) \to s(p)|  p\in \caM\}} whose elements can be taken from the dual basis of $\caM$ in $D\LL_n$.
The relations is derived from those of its trivial extension $\dtL$ of $\LL$ (see Proposition 2.2 of \cite{fp02} or Proposition 3.1 of \cite{g20}).
$\dtL$ is a graded self-injective algebra by taking the arrows of $\tQ$ as degree $1$ generators.
The relation set \eqqcn{retrel}{\trho=\rho\cup \rho_{\caM} \cup \rho_0} of $\dtL$ is a disjoint union of $3$ subsets: the relation $\rho$ for $\LL$, \eqqcnn{retrelm}{\rho_{\caM} =\{\beta_{p} \beta_{p'} |  p', p\caM \}} and $\rho_{0}$, the relation of $D\LL$ as a $\LL$-bimodule generated by  $Q_{1,\caM}$.
If $\dtL$ is quadratic, then the elements in $\rho_{0}$ are linear combinations of the elements of the forms $\xa \beta_p$ and $ \beta_{p'}\xa' $ for $p,p'\in \caM$ and $\xa,\xa'\in Q_1$.

The following Proposition is obvious.
\begin{pro}\label{triopp}
Let $Q$ be an $n$-properly-graded quiver and let $\LL= kQ/(\rho)$ be the algebra defined by $Q$.
Then

$\tQ$ is the bound quiver of the trivial extension $\dtL$ of $\LL$.

$\dtL^{op} = (\dtL)^{op}$ and its bound quiver is $\tQ^{op}$.

If $\dtL$ is $n$-translation algebra, so is $\dtL^{op}$.
\end{pro}

In fact, $\tQ$ is an $n$-translation quiver with trivial $n$-translation (see \cite{g16}).

If $\sigma$ is a grade automorphism of $\LL$, and $\dtsL$ is the twisted trivial extension of $\LL$ twisted by $\sigma$.
Then $\dtsL$ have the same quiver with $\dtL$, with relation set  \eqqcn{retrel}{\trho_{\sigma}=\rho\cup \rho_{\caM} \cup \rho_{0,\sigma},} where $$\rho_{0,\sigma}=\{\sum_{p_s\in \caM}a_s\xa_s\xb_{p_s}+\sum_{p_s\in \caM}b_s \xb_{p_s}\sigma(\xa'_s)| \sum_{p_s\in \caM}a_s\xa_s\xb_{p_s}+\sum_{p_s\in \caM}b_s \xb_{p_s}\xa'_s\in \rho_0\}.$$

Note that for an $n$-properly-graded quiver $Q$, its quadratic dual quiver $Q^{\perp}$ is an $n$-slice.
Let $\GG = k Q^{\perp}/(\rho^{\perp}) = \LL^{!,op}$ be the algebra defined by $Q^{\perp}$, $\GG$ is called an {\em $n$-slice algebra} if the trivial extensions $\dtL = \Delta\GG^{!,op}$ is an $n$-translation algebra.
By Proposition \ref{triopp}, we see that $ \Delta\GG^{ !} = \dtL^{op}$ is also an $n$-translation algebra.
So we have the following.

\begin{pro}\label{nslice:opp}
If $\GG$ is an $n$-slice algebra, $\GG^{op}$ is also an $n$-slice algebra.
\end{pro}

Let $\Pi(\GG)$ be the $(n+1)$-preprojective algebra of $\GG$.
We have that $\Pi(\GG) \simeq (\dtnL)^{!,op}$  \cite{g20} ($\GG$ is called a dual $\tau$-slice algebra there), here $\nu$ is the automorphism induced by the linear map on $kQ_1$ sending $\xa$ to $(-1)^n \xa$ for each arrow $\xa\in Q_1$.
It followings from \cite{g20} that $\Pi(\GG)$ is defined by the quadratic dual  quiver $\tQ^{\perp}$ of the returning quiver $\tQ$, with the relation set \eqqcn{relprepro}{\trho^{\perp} = \rho^{\perp} \cup \rho_{0,\nu}^{\perp},} where $\rho_{0,\nu}^{\perp}$ is a basis of the orthogonal subspace of $\rho_{0,\nu}$ in the space $k Q_1 Q_{\caM} + k  Q_{\caM} Q_1$ spanned by the paths of length $2$ with one arrow from $Q_1$ and the other one from $Q_{\caM}$.

Let $\LL$ be the algebra defined by the $n$-properly-graded quiver $Q $.
For simplicity, we assume that  $\dtL$ is quadratic.
The quiver $\zZ|_{n-1} Q$ is defined as the quiver  with vertex set \eqqcn{vertexzq}{(\zZ|_{n-1} Q)_0 =\{(u , t)| u\in Q_0, t \in \zZ\}} and arrow set \eqqcn{arrowzq}{\arr{rl}{(\zZ|_{n -1}Q)_1  = & \zZ \times Q_1 \cup \zZ \times Q_{1,\caM} \\ =& \{(\alpha,t): (u,t)\longrightarrow (u',t) | \alpha:u\longrightarrow u' \in Q_1, t \in \zZ\} \\& \quad \cup \{(\beta_p , t): (u', t) \longrightarrow (u, t+1) | p\in \caM, s(p)=u,t(p)=u'  \}.}}
$\zZ|_{n-1} Q$ is a  locally finite infinite quiver if $Q$ is locally finite.
The relation set of $\zZ|_{n-1} Q$ is defined as  \eqqcn{relzq}{\rho_{\zZ|_{n-1} Q} =\zZ \rho\cup \zZ \rho_{\caM} \cup \zZ\rho_0,} where
\eqqcn{rela}{\zZ \rho =  \{\sum_{s} a_s (\xa_s,t)(\xa'_s,t) |\sum_{s} a_s \xa_s \xa'_s \in \rho, t\in \zZ\},} \eqqcn{relb}{\zZ \rho_{\caM} =  \{(\beta_{p'},t)(\beta_p ,t+1)| \beta_{p'} \beta_{p}\in \rho_{\caM}, t\in \zZ\}} and \eqqcn{relc}{\arr{ll}{\zZ \rho_0 = &\{ \sum_{s'} a_{s'}  (\beta_{p'_{s'}}, t)  (\xa'_{s'}, t+1)   \sum_{s} b_s (\xa_s,t)  (\beta_{p_s} ,t)| \\ &\qquad \sum_{s'} a_{s'} \beta_{p'_{s'}} \xa'_{s'} + \sum_{s} b_s \xa_s \beta_{p_s} \in \rho_0 , t\in \zZ\}.}}

$\zZ|_{n-1} Q$ is a stable $n$-translation quiver with the $n$-translation defined by $\tau(i,m) =(i,m-1)$.
Similar to Proposition 5.5 of \cite{g16}, it is realized as a bound quiver of a smash product.
\begin{pro}\label{0nap:extendible1}
Let $\LL$  be an $n$-properly-graded algebra such that its trivial extension $\dtL$ is quadratic.
Then  the smash product $\dtL\#  k \zZ^*$ is a graded self-injective algebra with bound quiver $\zZ|_{n-1} Q $.
\end{pro}

Clearly, we have that $\dtL^{op}\#  k \zZ^*$is a graded self-injective algebra with bound quiver $\zZ|_{n-1} Q^{op} =(\zZ|_{n-1} Q)^{op}$.
Write $\zzs{n-1}Q^{\perp}$ for the quadratic dual quiver of $\zzs{n-1} Q$ and $\zzs{n-1}Q^{op,\perp}$ for the quadratic dual quiver of $\zzs{n-1} Q^{op}$.
Then $\zZ|_{n-1} Q^{op}$ is an $n$-translation quiver with $\tau^{-1}$ as its $n$-translation.

\medskip

Complete $\tau$-slice and $\tau$-slice algebra are introduced for stable $n$-translation quiver in \cite{g12}.
Let $\olQ=(\olQ_0, \olQ_1, \olrho)$ be an acyclic stable $n$-translation quiver with $n$-translation $\tau$, and assume that $\olQ$ has only finite many $\tau$-orbits.
Let $Q$ be a full subquiver of $\olQ$.
$Q$ is called a  {\em complete $\tau$-slice} of $\olQ$ if it has the following properties:
(a).for each vertex $v $ of $\olQ$, the intersection of the $\tau$-orbit of $v$ and the vertex set of $Q$ is a single-point set. (b). $Q$ is convex in $\olQ$.

Since for each pair $i,j$ of vertices in $\olQ_0$, we have that $e_i\olrho e_j \subset \olrho$.
Take $$\rho = \{x = \sum_p a_p p \in \olrho| s(p), t(p) \in Q_0 \} \subset \olrho,$$ we get a full bound subquiver $Q = (Q_0, Q_1,\rho)$ of $\olQ$ and call it a {\em complete $\tau$-slice} of the bound quiver $\olQ$.
The algebra defined by a complete $\tau$-slice $Q$ in a stable $n$-translation quiver $\olQ$ is called a {\em $\tau$-slice algebra}.

In fact, every acyclic $\tau$-slice algebra is obtained as a $\tau$-slice algebra of $\zzs{n-1} Q$.

\begin{pro}\label{znq}
Let $\olQ$ be an acyclic stable $n$-translation quiver and  let  $Q=(Q_0,Q_1,\rho)$ be an complete $\tau$-slice of $\olQ$, and let $\LL$ be the $\tau$-slice algebra defined by $Q$.
Then $Q$ is acyclic, $\olL \simeq \dtL \# kZ^*$ is the repetitive algebra of $\LL$ and we have that $\olQ \simeq \zZ|_{n-1} Q$.
\end{pro}
\begin{proof}
It follows from Theorem 6.10 of \cite{g12} that the trivial extension of all the $\tau$-slice algebras of $\olQ$ are isomorphic to $\dtL$.
By  Theorem 5.12 of \cite{g12}, we see that $\olL \simeq \dtL \# kZ^*$ is the repetitive algebra of $\LL$.
Now by Proposition \ref{0nap:extendible1}, the quiver of $\olL \simeq \dtL \# kZ^*$ is exactly $\zZ|_{n-1} Q$, hence we have that $\olQ \simeq \zZ|_{n-1} Q$.
\end{proof}

So we have the following consequence.

\begin{cor}\label{nslicezq}
Let $Q^{\perp}$ be an acyclic quadratic bound quiver and let $Q$ be its quadratic dual.
Let $\GG$ and $\LL$ be the algebras defined by $Q^{\perp}$ and $Q$, respectively.

Then $Q^{\perp}$ is an $n$-slice if and only if $Q$ is a complete $\tau$-slice in an acyclic stable $n$-translation quiver.

$\GG$ is an $n$-slice algebra if and only if $\LL$ is a $\tau$-slice algebra  and $\dtL$ is an $n$-translation algebra.
\end{cor}

If $\dtL$ is an $n$-translation algebra, then by Proposition \ref{triopp}, $\dtL^{op}$ is also an $n$-translation algebra.
Let $\olG = (\dtL\#  k \zZ^*)^{!,op}$ and $\olG^* = (\dtL^{op}\#  k \zZ^*)^{!,op}$.
Let $\olQ^{\perp}$ and $\olQ^{*,\perp}$ be  the bound quiver of $\olG$ and $\olG^*$ with relation sets $\olrho^{\perp}$ and $\olrho^{*,\perp}$, respectively.
Then $\olQ^{\perp} = \zzs{n-1} Q$, $\olQ^{*,\perp} = \zzs{n-1} Q^{op}$ as quiver and the relation sets  $\olrho^{\perp}$ and $\olrho^{*,\perp}$ are the quadratic dual relations as determined in \eqref{relationqd}.

In this way, an $n$-slice $Q^{\perp}$ is obtained from a complete $\tau$-slice $Q$ by taking the quadratic dual and an $n$-slice algebra $\GG$ is the quadratic dual of a $\tau$-slice algebra $\GG^{!,op}$.
We may regard  $n$-slices $Q^{\perp}$ and $Q^{op,\perp}$ as  $n$-dimensional generalization of slices in $\zzs{n-1} Q^{\perp} $ and in $\zzs{n-1} Q^{op, \perp}$, respectively.

An $n$-slice algebra is in fact an $n$-hereditary algebra defined in \cite{hio14}, that is, an $n$-representation-finite or an $n$-representation-infinite algebra.
The following theorem is Theorem 4.3  in \cite{gh21b}
\begin{thm}\label{suf_nec}
An $n$-hereditary algebra is $n$-slice algebra if and only if its $(n+1)$-preprojective algebra is $(q+1,n+1)$-Koszul for some $q$.
\end{thm}
\subsection{Example: Path algebra and $\zZ Q$.}
\label{sub:eg.here}
If $Q$ is an acyclic quiver, take the set of all the path of length $2$ as $\rho$, then $Q$ is a $1$-properly-graded quiver.
So its quadratic dual is the quiver $Q$ with empty set as the relation set.
The bound quiver algebra $\LL =kQ/(\rho)$ is $1$--properly-graded algebra and its quadratic dual $\GG =kQ$ is exactly the path algebra.
The quiver $\tQ$ of the trivial extension $\dtL$ of $\LL$ is the double quiver obtained by adding an opposite arrow $\xa^*$ for each arrow $\xa$ in $Q_1$, and the relation $\trho$ defining $\dtL$ is the set consisting of all the paths of length $2$ except those of the form $\xa^*\xa$ or $\xa \xa^*$, and the commutative relations for each pair of paths of the form $\xa^*\xa$ or $\xb \xb^*$ starting at the same vertex.
The trivial extension $\dtL$ and twisted trivial extensions $\dtnL$ are a stable $1$-translation algebras (weakly symmetric algebra with vanishing radical cube).
The quadratic dual of the twisted trivial extension $\dtnL$ is  the preprojective algebra $\Pi(Q)$ of $kQ$.
In this case, $\zZ|_0 Q = \zZ Q$ is the classical one, and the elements in $\rho_{\zZ|_0 Q}^{\perp} $ can be chosen as the mesh relations.

$Q$ is a complete $\tau$-slice of $\zZ|_0 Q$, and $\rho^{\perp} =\emptyset$.
So the path algebra $kQ$ is a $1$-slice algebra.

\subsection{Example: Auslander-Reiten quiver and $\zzs{1}Q$}
\label{sub:eg.mesh}
We can always associate a $2$-slice algebra to the Auslander-Reiten quiver of a representation-directed algebra.

If $Q^{\perp}= (Q_0,Q_1,\rho^{\perp})$ is an Auslander-Reiten quiver of a representation-directed algebra $A$, then  $\rho^{\perp}$ is the set of mesh relations.
The algebra $\GG$ defined by $Q^{\perp}$ is the Auslander algebra of $A$.
Let $\LL$ be the quadratic dual of $\GG$, then $\LL$ is an algebra with vanishing radical cube and hence is a $1$-translation algebra \cite{g02} and its bound quiver $Q= (Q_0, Q_1, \rho)$ is a $1$-translation quiver.
It is easy to see that $Q$ is a $2$-properly-graded and $\LL$ is a $2$-properly-graded algebra.
The relation set $\rho$ can be taken as a basis of the orthogonal space of $\rho^{\perp}$ in $kQ_2$.

In this case, $\tL$ is $2$-translation algebra, and $Q$ is a complete $\tau$-slice in stable $2$-translation quiver $\zZ|_1 Q$.
So the quadratic dual $\GG$ of $\LL$ is a $2$-slice algebra.

Let $Q^{\perp}$ be the Auslander-Reiten quiver of path algebra of quiver of type $A_4$ with linear orientation,
$$
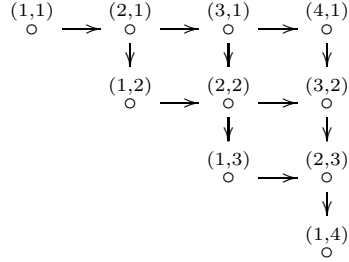
\xymatrix@C=0.5cm@R0.3cm{
&& \stackrel{(1,1)}{\circ} \ar[r] &\stackrel{(2,1)}{\circ} \ar[r]\ar[d] &\stackrel{(3,1)}{\circ} \ar[r]\ar[d] &\stackrel{(4,1)}{\circ}\ar[d]\\
&&              &\stackrel{(1,2)}{\circ} \ar[r]   &\stackrel{(2,2)}{\circ} \ar[d] \ar[r]       &\stackrel{(3,2)}{\circ}\ar[d]  &{} \\
&&              & &\stackrel{(1,3)}{\circ} \ar[r]       &\stackrel{(2,3)}{\circ}\ar[d]  &{} \\
&&                           &&      &\stackrel{(1,4)}{\circ}  &{} \\
}
$$
Then the quadratic dual  $Q$ of $Q^{\perp}$ is a $1$-translation quiver with commutative relation for each square and zero relation for each pair of successive arrows heading to the same direction.
The $1$-translation is the usual translation with $\tau (i,t) = (i,t-1)$ when both $(i,t)$ and $ (i,t-1)$ are in the quiver.
So $Q$ is a $2$-properly-graded quiver.
The returning arrow quiver $\tQ$ is
$$
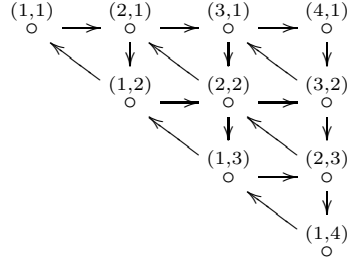
\xymatrix@C=0.5cm@R0.3cm{
&& \stackrel{(1,1)}{\circ} \ar[r] &\stackrel{(2,1)}{\circ} \ar[r]\ar[d] &\stackrel{(3,1)}{\circ} \ar[r]\ar[d] &\stackrel{(4,1)}{\circ}\ar[d]\\
&&              &\stackrel{(1,2)}{\circ} \ar[r]\ar@[red][ul]  &\stackrel{(2,2)}{\circ} \ar[d] \ar[r]\ar@[red][ul]      &\stackrel{(3,2)}{\circ}\ar[d]\ar@[red][ul] &{} \\
&&              & &\stackrel{(1,3)}{\circ} \ar[r]\ar@[red][ul]      &\stackrel{(2,3)}{\circ}\ar[d]\ar@[red][ul] &{} \\
&&                           &&      &\stackrel{(1,4)}{\circ}\ar@[red][ul] &{} \\
}
$$
\normalsize
The full subquiver formed by three neighbour slices in the quiver $\zzs{1}Q$ looks like
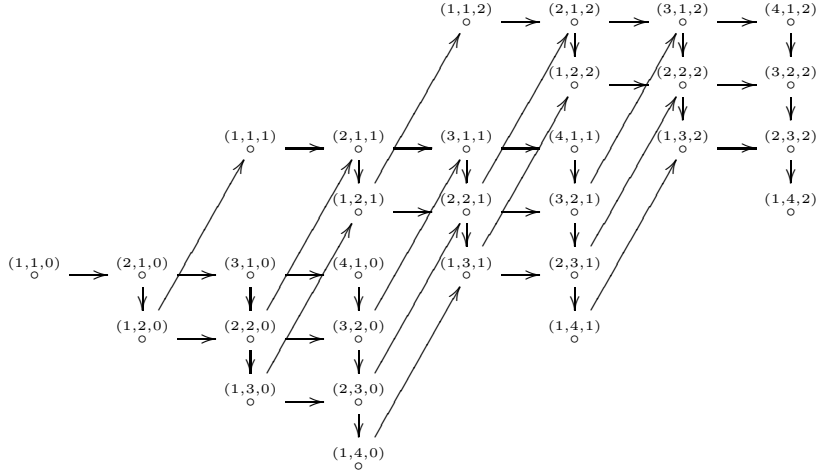
\tiny$$
\xymatrix@C=0.5cm@R0.3cm{
&&&&&
&&&\\
&&&
&&& \stackrel{(1,1,2)}{\circ} \ar[r] &\stackrel{(2,1,2)}{\circ} \ar[r]\ar[d] &\stackrel{(3,1,2)}{\circ} \ar[r]\ar[d] &\stackrel{(4,1,2)}{\circ}\ar[d] &&&&\\
&&&
&&
&&\stackrel{(1,2,2)}{\circ} \ar[r] &\stackrel{(2,2,2)}{\circ} \ar[r] \ar[d]&\stackrel{(3,2,2)}{\circ}\ar[d]&&&&&\\
&&&&  \stackrel{(1,1,1)}{\circ} \ar[r] &\stackrel{(2,1,1)}{\circ} \ar[r]\ar[d] &\stackrel{(3,1,1)}{\circ} \ar[r]\ar[d] &\stackrel{(4,1,1)}{\circ}\ar[d]
&\stackrel{(1,3,2)}{\circ} \ar[r]     &\stackrel{(2,3,2)}{\circ}\ar[d]
 &\\
&&
&& &\stackrel{(1,2,1)}{\circ} \ar[r]\ar@[red][uuur]  &\stackrel{(2,2,1)}{\circ}\ar[r]\ar@[red][uuur] \ar[d]&\stackrel{(3,2,1)}{\circ}\ar[d]\ar@[red][uuur]
 &   &\stackrel{(1,4,2)}{\circ}&{} \\
&&\stackrel{(1,1,0)}{\circ} \ar[r] &\stackrel{(2,1,0)}{\circ} \ar[r]\ar[d] &\stackrel{(3,1,0)}{\circ} \ar[r]\ar[d] &\stackrel{(4,1,0)}{\circ}\ar[d]&\stackrel{(1,3,1)}{\circ} \ar[r]\ar@[red][uuur]  &\stackrel{(2,3,1)}{\circ} \ar[d]\ar@[red][uuur]&&\\
&& &\stackrel{(1,2,0)}{\circ} \ar[r]\ar@[red][uuur]  &\stackrel{(2,2,0)}{\circ}\ar[r]\ar@[red][uuur] \ar[d]&\stackrel{(3,2,0)}{\circ}\ar[d]\ar@[red][uuur]
&&      \stackrel{(1,4,1)}{\circ}\ar@[red][uuur] &{} \\
&&&&\stackrel{(1,3,0)}{\circ} \ar[r]\ar@[red][uuur]  &\stackrel{(2,3,0)}{\circ} \ar[d]\ar@[red][uuur]&&\\
&&&&      &\stackrel{(1,4,0)}{\circ}\ar@[red][uuur] &{} \\
}
$$
\normalsize
This quiver $\zzs{1}Q^{\perp}$ is the same as the cylinder constructed for $\mathcal U^{(2)}$ in Example 6.13 of \cite{i11}

\section{Bound path category of a dual stable $n$-translation quiver}
\label{sec:bpc}

One important feature of the $n$-translation algebra is that it induces $n$-almost split sequences in the category of finitely generated projective modules of its quadratic dual via Koszul complexes, under certain conditions.
Now we study the bound path categories related to an acyclic $n$-translation quiver.

Throughout this section, we assume that $\olQ= (\olQ_0,\olQ_1, \overline{\rho} )$ is an acyclic stable $n$-translation quiver with $n$-translation $\tau$ such that $\olL \simeq k\olQ/(\olrho)$ is an $n$-translation algebra, and $\olQ$ has only finite many $\tau$-orbits.
Then $\olL \simeq k\olQ/(\olrho)$ is $(n+1,q)$-Koszul algebra for some integer $q\ge 2$ or $q=\infty$ \cite{g16}.
Let $\olG$ be the quadratic dual of $\olL$ with bound quiver $\olQ^{\perp} =(\olQ_0,\olQ_1, \overline{\rho}^{\perp} )$.

Let $\bcaG = \add(\olG)$ (respectively, $\bcaL = \add(\olL)$) be the category of finite generated projective $\olG$-modules (respectively, $\olL$-modules).
$\{e_i| i\in \olQ_0\}$ is a complete set of orthogonal primitive idempotents.
We will use the same notations for the orthogonal primitive idempotents when no confusion occurs.
Write $\ind \caC$ for the set of iso-classes of indecomposable objects in a category $\caC$, we have $$\ind \bcaG =\{ \olG e_i| i\in \olQ_0\}, \quad \mbox{and} \quad \ind \bcaL =\{ \olL e_i| i\in \olQ_0\}.$$
Since $\olQ$ is acyclic, $\bcaG$ and $\bcaL$ are Hom-finite Krull-Schmidt and we have $\End X \simeq k$ for each indecomposable object $X$ in these categories.

Recall that an additive category $\caA$, equipped with a function $deg: \ind \caA  \to \zZ$, assigning an indecomposable object to its degree, is called an {\em Orlov category} if it satisfies the following conditions:(1) All Hom-spaces in $\caA$ are finite-dimensional; (2) For any $X \in \ind \caA $, $\End_{\caA}(X) \simeq k$; (3) If $X, X' \in \ind \caA $ with $deg(X) \le deg(X' )$ and $X\not\simeq X'$, then $\hhm_{\caA}(X,X' ) = 0$ (see \cite{ar13}).

By Proposition \ref{pathalgcat}, $\add \olL$ is equivalent to the bound path category $\caP(\olQ^{op})$ and $\add \olG$ is equivalent to the bound path category $\caP(\olQ^{op,\perp})$.
Take a $\tau$-slice $Q$ of $\olQ^{op}$, then by Proposition \ref{znq}, $\olQ^{op}\simeq \zzs{n-1} Q$.
Since $Q=(Q_0,Q_1,\rho_Q)$ is an acyclic bound quiver, we can define an order on $Q$ by indexing its vertices as $Q_0 = \{u_1,\cdots,u_l\}$ such that there is no path from $u_t $ to $u_{t'}$ if $t<t'$.
Reindex the vertices in $\olQ^{op}$ as \eqqcn{indexolq}{\{(u_t,r)| 1\le t\le l, r\in \zZ\},} using $\zzs{n-1} Q $.
For each vertex $i = (u_t,r) \in \olQ^{op}_0$, define \eqqcn{dgq}{d_Q(i)= d_Q (u_t,r) = -rl +t.}
Then there is no path in $\olQ^{op}$ from $i=(u_t,r)$ to $i'=(u_{t'}, r')$ if $d_Q (i)< d_Q (i')$.
We call this order {\em the order of $\olQ^{op}$ induced by the path order of $Q_0 = \{u_1,\cdots,u_l\}$}.
We have that $1\le d_Q(i) \le l$ for $i\in Q_0$, and the full subquiver $Q[r]$ with vertex set \eqqcn{vertexGr}{\{i\in \olQ^{op}_0 | rl+1\le d_Q(i) \le (r+1) l\}} are complete $\tau$-slices isomorphic to $Q$, and $Q[0] = Q$.
It is obvious that if there is a path from $i $ to $j$ in $\olQ^{op}$, then $d_Q(i) \ge d_Q(j)$.
Define the degrees for $\bcaG$ and $\bcaL$ by setting $$deg(\olG e_i)= d_Q(i)  \mbox{ and } \deg(\olL e_i) =d_Q(i),$$ then we immediately have the following.

\begin{pro}\label{orlov}
With the degrees defined above, $\bcaG$ and $\bcaL$ are both Orlov categories.
\end{pro}

Write \eqqcn{caGt}{\caG[t]= \add\{\olG e_j | j\in Q_0[t]\}.}
Denote by $\mathbb Z^+$ the set of nonnegative integers.
Let  $\olQ^{+}$  and $\olQ^{-} $  be the full subquivers of $\zZ|_{n-1} Q^{\perp}$ with the vertex sets \eqqcn{vertexQpm}{\olQ^{+}_0 = \{\tau^t i | i\in Q_0, t\in \zZ^+ \} \mbox{ and } \olQ^{-}_0 = \{\tau^{-t} i | i\in Q_0, t\in \zZ^+ \},} respectively.
Let $\LL = k Q^{op}/(\rho_{Q}^{op})$ and let $\GG = kQ^{op}/(\rho_{Q}^{op,\perp})$.
Write $\add(\GG)$ (respectively, $\add(\LL)$) for the category of finite generated projective $\GG$-modules (respectively, $\LL$-modules), then $\caG[t]$ are equivalent to $\caG =\add \GG$ for all $t$.
$\{ \GG e_i| i\in Q_0\}$ (respectively, $\{ \LL e_i| i\in Q_0\}$) is a set of non-isomorphic indecomposable objects in $\add(\GG)$ (respectively, $\add(\LL) $).
Write \eqqcn{bcaGpm}{\bcaG^+ =\add \{ \olG e_j| j\in \olQ^{+}_0\}, \mbox{ and } \bcaG^- =\add \{ \olG e_j| j\in \olQ^{-}_0\}.}
Then for $\olG e_j\in \bcaG^+$, we have $u\in Q_0$ and $t\in \zZ^+$ such that $j= \tau^{t} u$, and for $\olG e_{j'}\in \bcaG^-$, we have $u'\in Q_0$ and $t'\in \zZ^+$ such that $j'= \tau^{-t'} u'$.
If $u$ is a vertex of $Q_0$, we identify the vertex $u$ with $(u,0)$ in $\zzs{n-1} Q$, then  $\tau^t u = (u, -t)$ for $t\in \zZ$.

\medskip

Let $Q$ be a full subquiver of $\olQ^{op}$, a full subquiver $Q'$ of $Q$ is called  {\em initial} in $Q$ provided that it is convex and for each  arrow $\xa$ from $i$ to $j$ in $Q$, if  $j \in Q'_0$, then $i \in Q'_0$.
A {\em terminating } subquiver of $Q$ is defined dually.

\begin{pro}\label{relation:local}
Let $\caC$ be a Hom-finite Krull-Schmidt $k$-category, and assume that   $\End_{\caC} X/J_{} \End_{\caC} X \simeq k$ for each indecomposable object $X$ in $\caC$.
If there is a correspondence $F$ from the objects of $\caC$ to the objects of $\bcaG$ satisfies the conditions
\begin{itemize}
\item[(i).] If $X \in \ind(\caC)$, then $F(X) = \olG e_i$ for some $i\in \olQ^{op}_0$.

\item[(ii).] The full subquiver $Q = Q_{\caC}$ of $\olQ^{op}$ with vertex set $$Q_{\caC,0} =\{i \in \olQ^{op}_{0}| \olG e_i \simeq F(X) \mbox{ for some } X\in \ind(\caC)\}$$ is a $\tau$-mature subquiver of $\olQ^{op}$ containing an initial (respectively, a terminating) full subquiver $Q'$, and $Q'$ contains a complete $\tau$-slice $Q''$ of $\olQ^{op}$.
\end{itemize}
    Let \eqqcn{caGCGp}{\caG(\caC) = \add \{\olG e_j | j\in Q_{\caC,0}\} \mbox{ and }\caG' = \add \{\olG e_j | j\in Q_0'\}.}
\begin{itemize}
\item[(iii).] $F$ induces an equivalence from $\caC' =\add\{X| F(X) = \GG e_j, j\in Q'_0\}$ to $\caG'$.

\item[(iv).] If \eqqcn{sinksourceseq}{M_{n+1} \lrw \cdots \lrw M_t \lrw \cdots \lrw M_0} is a sink sequence (respectively, source sequence) in $\caC$ then \eqqcn{sinksourceseqimg}{F(M_{n+1}) \lrw \cdots \lrw F( M_t) \lrw \cdots \lrw F(M_0)} is a sink sequence (respectively, source sequence)  in $\caG(\caC)$.
\item[(v).] If $X, Y$ are indecomposable objects in $\caC$ such that $F(X) =\olG e_i, F(Y) =\olG e_j$ and $d_{Q''} (i) < d_{Q''} (j)$, then $\hhom_{\caC}(X,Y) =0$.
\end{itemize}
Then $F$ defines an equivalence of $\caC$ with the full subcategory $\caG(\caC)$ of $\bcaG$.
\end{pro}

\begin{proof}
By Proposition \ref{pathalgcat}, we have that $\caP(\olQ^{op,\perp}) \simeq \bcaG$, so there is a functor $H$ from the path category $\bfP(\olQ^{op,\perp})$ of $\olQ^{op}$ to $\bcaG$ with kernel the ideal generated by $\olrho^{op,\perp}$.

Since $\caC$ is Krull-Schmidt and Hom-finite, we have that $\End_{\caC} M$ is finite dimensional for each object $M$ of $\caC$.
Since $\End_{\caC} X/J_{} \End_{\caC} X \simeq k$ for each indecomposable object $X$ in $\caC$, so $\End_{\caC} M/J_{} \End_{\caC} M \simeq k^s$ is a direct product of $k$ for each basic object $M$ in $\caC$.

Note that $Q_{\caC}$ is convex, so for any $i,j\in Q_{\caC,0}$, a path from $i$ to $j$ in $\olQ^{op}$ is also a path in $Q_{\caC}$.
Thus we may require that the relation set \eqqcn{rhocaC}{{\rho_{\caC}}^{\perp} = \{e_j x e_i| x \in \olrho^{op,\perp}, i,j \in Q_{\caC,0}\}} is a subset of $\olrho^{op,\perp}$.
The full subcategory \eqqcn{caGC}{ \caG(\caC) = \add \{\olG e_j | j\in Q_{\caC,0}\}} is the bound path category of  $Q_{\caC}=(Q_{\caC,0}, Q_{\caC,1}, {\rho_{\caC}}^{\perp})$, by Proposition \ref{pathalgcat}.
$\caG'$ is a full subcategory of $\caG(\caC)$, which is equivalent to the bound path category of ${Q'}^{\perp} = ( Q'_0, Q'_1, {\rho'}^{ \perp})$, since $Q'$ is convex.

By (iii), $F$ induces an equivalent functor, write again as $F$, from $\caC'$ to $\caG'$.
Let $G $ be the quasi inverse of $F$ from $\caG'$ to $\caC'$.
Then we have a functor $\widehat{G}= G H$ from the path category $\bfP(Q')$ to $\caC'$ sending vertices to the indecomposables and arrows to the representatives of a basis of $J_{\caC}(X,Y)/J_{\caC}^2(X,Y)$ with kernel the ideal generated by ${\rho'}^{\perp}$.

We consider the case that $Q'$ is an initial full subquiver of $Q$, the other case is proven dually.

Note that $Q'$ contains a complete $\tau$-slice of $\olQ^{op}$, if for all $t$, the images of the action of $\tau^{-t}$ on the complete $\tau$-slice in $Q'$ remain in $Q'$, then $Q = Q'$ and the proposition holds.
Otherwise, let $Q''$ be a complete $\tau$-slice of $Q$ in $Q'$ such that $\tau^{t} Q'' $ are in $Q'$ for $t\ge  0$, and $\tau^{-1} Q''$ is not in $Q'$.
Consider the order of $\olQ^{op}$ induced by a path order of $Q''_0 =\{ u_1,\ldots, u_l\}$.
Since $Q'$ is an initial full subquiver of $Q$, for each indecomposable $X$ in $\caC$, if $X$ is not in $\caC'$, then we have that $F(X) \simeq \olG e_{\tau^{-t} u}$ for $u\in Q_0''$ and some $t > 0$, by (i).

For an $r\ge 0$, define  \eqqcnn{QCr}{Q_{\caC}(r)} as the full subquiver of $\olQ$ with vertex set  \eqqcn{vertexQCr}{Q_{\caC,0}(r)= \{i \in \olQ_0 | d_{Q''}(i)\ge -r\}\cup Q'_0.}
Then $Q'= Q_{\caC}(0)$.
Clearly, each $Q_{\caC}(r)$ is convex and the vertex set \eqqcn{vertexQC}{Q_{\caC,0} = \cup_{r} Q_{\caC,0} (r).}
Let \eqqcn{bcaGrCr}{\bcaG(r) = \add\{\olG e_{j}| j\in Q_{\caC,0}(r) \} \mbox{ and } \caC(r) = \add \{X| F(X) \simeq \olG e_{j}, j\in Q_{\caC,0}(r) \}.}
$\bcaG(r)$ is the bound path category of the bound quiver $Q_{\caC}(r) = (Q_{\caC,0}(r), Q_{\caC,1}(r), \rho_{\caC}^{\perp}(r) )$.
We have that $Q_{\caC}(0) = Q'$ and $\caC' $ is equivalent to $\caG'$, that is, $\caC(0)$ is equivalent to $\bcaG(0)$, and both are equivalent to the bound path category $\caP(Q_{\caC}(0))$.
We regard $\widehat{G}$ as a functor from $\mathbf P(Q_{\caC}(0))$ to $\caC(0)$ which induces an equivalence from $\caP(Q_{\caC}(0)$ to $\caC(0)$.

We extend functor $\widehat{G}$ inductively, assume that $\widehat{G} (j) = X_j$ is defined for the path category $\bfP(Q_{\caC}(r-1))$ of the quiver $Q(r-1)$ such that $F \widehat{G} (j) = \olG e_j$, and $\widehat{G}$  induces an equivalence  $G$ from $\bcaG(r-1)$ to $\caC(r-1)$ such that  $$  \hhom_{\bcaG} (\olG e_j, \olG e_{j'}) = \hhom_{\caC} (G(\olG e_j), G(\olG e_{j'})) = \hhom_{\caC} (X_j, X_{j'}) $$ for $\olG e_j, \olG e_{j'}$ in $\bcaG(r-1)$.

If $r\ge 1$, and assume that $d_{Q''}(i) = -r$.
If $i \in Q'_0$, then $Q_{\caC}(r) = Q_{\caC}(r-1)$ and we are done.
If $i \not\in Q'_0$, then $i=\tau^{-s}u$ for some $u \in Q''_0$ and $\tau^{1-s}u \in Q_{\caC,0}(r-1)$ since $Q''$ is a complete $\tau$-slice and $Q$ is convex.
So $\olG e_{\tau^{1-s}u } \in \bcaG(r-1)$.

By (v), we have that $\hhom_{\caC}(X_i, X_j) =0$ if $F(X_i) =\olG e_i$ and $F(X_{j}) =\olG e_{j}$ for any $i\neq j \in Q_{\caC,0}(r)$, since $d_{Q''}(i) < d_{Q''}(j)$.
So \eqqcn{homeq}{\hhom_{\caC}(X_i, X_j) = \hhom_{\bcaG}(\olG e_i, \olG e_j)=0,} since there is no path in $\olQ^{op}$ from $i$ to $j$.

Now we consider $\hhom_{\caC}(X_j, X_i)$.

Since $\olG$ is the quadratic dual of an $n$-translation algebra, we have a sink sequence
\eqqc{sinkinG}{\olG e_{\tau^{1-r}u } \lrw \cdots \lrw \bigoplus_{(j',2)\in H_{i}}  ( \olG e_{j'} )^{\mu_i(j',2)} \slrw{F(\psi)} \bigoplus_{(j',1)\in H_{i}}  ( \olG e_{j'} )^{\mu_i(j',1)} \slrw{F(\phi)} \olG¡¡e_{i} } in $\bcaG$, by Lemma 7.1 of \cite{g16}.
So by condition (iv), we have a sink sequence \eqqc{sinkinC}{ X_{\tau^{1-r}u} \lrw \cdots \lrw \bigoplus_{(j',2)\in H_{i}}  X_{j'}^{\mu_i(j', 2)} \slrw{\psi} \bigoplus_{(j',1)\in H_{i}}  X_{j'}^{\mu_i(j' ,1)} \slrw{\phi} X_{i} } in $\caC$ with all terms except the last one in $\caC(r-1)$.
So we have an exact sequence
\eqqc{exactseq1}{\arr{ll}{ \Hom_{\caC}(X, X_{\tau^{r-1}u}) \lrw \cdots \\ \qquad \lrw \bigoplus_{(j',t)\in H_{i}}  \Hom_{\caC}(X, X_{j'})^{\mu_i(j', t)} \lrw \cdots \\ \qquad \lrw\bigoplus_{(j',2)\in H_{i}}  \Hom_{\caC}(X, X_{j'})^{\mu_i(j', 2)} \\ \qquad \lrw \bigoplus_{(j',1)\in H_{i}}  \Hom_{\caC}(X, X_{j'})^{\mu_i(j', 1)}\\ \qquad\qquad \lrw J_{\caC}(X, X_{i} )\lrw 0 }}
for any $X \in \caC$.

Now take $X= X_j$ for vertex $j$ with $(j,1)\in H_i$ in \eqref{exactseq1}, then for $t\ge 1$ and $(j',t)\in H_i$ with $j\neq j'$, we have \eqqcn{homeqzero}{\Hom_{\caC}(X_j, X_{j'}) = \Hom_{\bcaG}(\olG e_j, \olG e_{j'}) =0,} by inductive assumption.
So we get exact sequence \eqqcn{eseq}{0 \lrw \Hom_{\caC}(X_j, X_{j})^{\mu_i(j, 1)} \slrw{\phi^*} J_{\caC}(X_j, X_{i} ) \lrw 0,} and thus \eqqcn{homeqc}{ \arr{rl}{ \Hom_{\bfP(Q(l))}(j,i) \simeq & (e_i \olG_1 e_j)  \simeq \Hom_{\bcaG}(\olG e_j,\olG e_i)\simeq  \Hom_{\caC} (X_j, X_i)\\ \simeq & J_{\caC}(X_j,X_i) \simeq J_{\caC}(X_j,X_i)  / J^2_{\caC} (X_j, X_i) .}}
Set $\widehat{G}(i) =X_i$.
This shows that \eqqcn{dimeq}{\dmn_k J_{\caC}(X_j,X_i)  / J^2_{\caC} (X_j, X_i) =\mu_i(j, 1) } is exactly the number of the arrows from $j$ to $i$, and the images \eqqcn{arrowsbasis}{\alpha_1(ji), \ldots, \alpha_{\mu_i({j,1})}(ji)} of the standard basis of $\Hom_{\caC}(X_j, X_{j})^{\mu_i(j, 1)}$ under $\phi^*$ form a basis of the space $J_{\caC} (X_j,X_i)  / J^2_{\caC} (X_j, X_i)$.
By Lemma 6.4 of \cite{i11} and Proposition 3.5 of \cite{birsm11}, $\widehat{G}$ is extended to a functor from the path category $\bfP(Q(r))$ to $\caC(r) = \add \{\widehat{G}(j)| j \in Q(r)\}$  sending the arrows from $i$ to $j$ to elements in $J_{\caC} (X_j,X_i)$ whose images form a basis of $J_{\caC} (X_j,X_i) / J^2_{\caC} (X_j, X_i)$.

Now take $X= X_j$ for vertex $j$ with $(j,2)\in H_i$ in \eqref{exactseq1}, since \eqref{sinkinG} and \eqref{sinkinC} are sink sequences in $\bcaG(r)$ and $\caC(r)$, respectively.
We get the following commutative diagram with exact rows,
\small
$$\arr{rcl}{0 \lrw \Hom_{\caC}(X_j, X_{j})^{\mu_i(j ,2)} &\slrw{\psi^*} \bigoplus \Hom_{\caC}(X_j, X_{j'})^{\mu_i(j', 1)} \slrw{\phi^*} & J_{\caC}(X_j, X_i )\lrw 0\\
\downarrow \simeq \qquad&\downarrow \simeq & \qquad\downarrow \\
0 \lrw \Hom_{\bcaG}(\olG e_j, \olG e_{j})^{\mu_i(j, 2)} &\slrw{F(\psi)^*} \bigoplus  \Hom_{\bcaG}(\olG e_j, \olG e_{j'})^{\mu_i(j', 1)} \slrw{F(\phi)^*} & J_{\bcaG}(\olG e_j, \olG e_i )\lrw 0.\\
}$$
\normalsize
By inductive assumption, the first two vertical maps are  isomorphisms.
So the last vertical map is also isomorphism, by the five lemma.
Comparing the images of $\phi^*\psi^*$ and $F(\phi)^*F(\psi)^*$, it follows from (2) of Lemma \ref{category:relation} that the kernel of $\widehat{G}$ is generated by $e_j\rho^{\perp}_{\caC}(r) e_i$, the same as for the bound quiver category $\bcaG(r)$.

Thus both $\bcaG(r)$ and $\caC(r)$ are equivalent to the quotient category of the path category $\bfP(Q_{\caC}(r))$ modulo the ideals generated by the same relation set,  so they are equivalent.
Especially, we have that \eqqcn{homeqb}{\Hom_{\caC}(X_j, X_{j'}) \simeq  \Hom_{\bcaG}(\olG e_j, \olG e_{j'})} if $F(X_j) =\olG e_j, F(X_{j'}) =\olG e_{j'},$ for any $j, j' \in Q_{\caC}(r)$.

This proves that $F$ induces an equivalence from $\caC$ to $\caG(\caC)$, by induction.
\end{proof}

This Proposition also  shows that $\caC$ is equivalent to the bound path category of the full bound subquiver $Q_{\caC}^{\perp}$ of $\olQ^{op,\perp}$.
And $Q_{\caC}^{\perp}$ is the Auslander-Reiten quiver of $\caC$.

The following corollary is obvious.

\begin{cor}\label{orlovC}
With the induced order on $\caC$ as in the proof of Proposition \ref{relation:local}, $\caC$ is an Orlov category and  $F$ preserves the order for the indecomposable objects.
\end{cor}

\section{$n$-preprojective component, $n$-preinjective component and truncations of $\zZ|_{n-1} Q^{op,\perp}$}
\label{sec:ttn}

Now we study the $n$-preprojective and $n$-preinjective components of an $n$-slice algebra $\GG$ defined by an acyclic bound quiver $Q^{\perp}=(Q_0,Q_1,\rho^{\perp})$.
Let $\olG$ be the algebra defined by the bound quiver $\zZ|_{n-1} Q^{\perp} $, we have that $\GG \simeq  e \olG e $ for $e= \sum_{j\in Q_0} e_j $.

\begin{lem}\label{globaldim}
$\GG$ is Koszul algebra of global dimension $ n$.
\end{lem}
\begin{proof}
Let $\LL$ be the quadratic dual of $\GG$, with the bound quiver $Q=(Q_0,Q_1,\rho)$.
By Proposition 6.3 of \cite{g20}, the indecomposable projective modules of $\LL$ have the Loewy length at most $ n+1$, and the  length of longest bound paths in $Q$ are exactly $n$.
By Proposition 6.5 of \cite{g20}, $\LL$ is a Koszul algebra.
So as its quadratic dual, $\GG$ is a Koszul algebra of global dimension $n$, by Theorem 2.10.2 and Theorem 2.6.1 of \cite{bgs96}.
\end{proof}

For a finite dimensional algebra $\GG$, the $n$-Auslander-Reiten translations of $\GG$ modules are introduced by Iyama (see \cite{i007,i07}), $$\tau_n = D\tr \OO^{n-1},\quad and \quad \tau_n^{-} =\tr D\OO^{1-n},$$ with the convention that $\ttn{0}=\ttn{-0}=\identity$.
Since $\gl \GG =n$, we have $\ttn{-1}  = \Ext_{\GG}^n(D -, \GG)$ and $\ttn{}  = D\Ext_{\GG}^n(-, \GG)$.

Let \eqqcn{}{\arr{lc}{\caM^+ =\caM^+ (\GG) =\add \{ \ttn{t} D\GG | t \ge 0\} ,& \mbox{and}\\  \caM^- =\caM^- (\GG) =\add \{ \ttn{-t} \GG | t \ge 0\}.}}
Call $\caM^+$ the {\em the $n$-preinjective components of $\Gamma$} and call its object $n$-preinjective module, call $\caM^-$ the {\em the $n$-preprojective  components of $\Gamma$} call its object $n$-preprojective module.
$\GG$ is called {\em $\ttn{}$-finite} if $ \caM^+$, or equivalently, $ \caM^-$ is of finite type, that is, has only finite many non-isomorphic indecomposable objects \cite{io13}.

Our aim in this section is to show that the Auslander-Reiten quiver of $n$-preprojective component and $n$-preinjective component in the category of  $\GG$-modules  are both truncations of $\zZ|_{n-1} Q^{op,\perp}$.

Note that $\caM^+$ and $\caM^-$ are  both hom-finite Krull-Schmidt categories, and indecomposables have the form $\ttn{t}D e_u\GG$ and $\ttn{-t} \GG e_u$, respectively, for integer $t\ge 0$ and $u\in Q_0$.
Especially, $\End_{} X \simeq k$ for any indecomposable object $X$ in $\caM^{\pm}$.
Thus for $t\ge 0$, the correspondence \eqqcn{}{\Theta^-: \ttn{-t}\GG e_u \to \olG e_{\tau^{-t} u}} assigns objects  from $\caM^-$ to objects in $\bcaG^-$,  and the correspondence \eqqcn{}{\Theta^+: \ttn{t} D (e_u \GG)\to \olG e_{\tau^{t} u}} assigns objects from $\caM^+$ to objects in $\bcaG^+$.

Since $Q^{op}$ is acyclic,  consider the order of $\zzs{n-1} Q^{op}$ induced by the path order of $Q^{op} =\{u_1,\ldots,u_m\}$.
This is the order induced by the map $d_{Q^{op}}:(u_t,r) \to -rl+t$ from $\zzs{n-1}Q^{op}_0$ to the integer $\zZ$.

For $m\ge 0$ and  $u \in Q^{op}_0$, let $Q^+_{(u,m)}$ be the full subquiver of $\zZ|_{n-1} Q^{op}$ with vertex set $$(Q^+_{(u,m)})_0 =\{(u',t)|0\le d_{Q^{op}} (u',t) \le d_{Q^{op}}(u,m)\},$$ and let $Q^-_{(u,m)}$ be the full subquiver of $\zZ|_{n-1} Q$ with vertex set $$(Q^-_{(u,m)})_0 = \{(u',t)|l \ge d_{Q^{op}}(u',t) \ge d_{Q^{op}}(u,m)\}.$$
Let
\eqqcn{cats}{\arr{l}{\caM^+_{(u,m)} = \{\ttn{t}D\GG e_{u'}| (u',-t)\in (Q^+_{(u, m)})_0\} \\
\caM^-_{(u,m)} = \{\ttn{-t}\GG e_{u'}| (u',t)\in (Q^-_{(u, m)})_0\}\\
\bcaG^\pm_{(u, m)} = \{\olG e_{\tau^{-t} u'}| (u',t)\in (Q^\pm_{(u, m)})_0\}.}}
For $X= \ttn{t} D\GG e_u$ in $\ind \caM^+$, define \eqqcn{}{ \deg(X) = d_{Q^{op}}(u,-t),} and for $ X= \ttn{-t} \GG e_u$ in $\ind\caM^-$, define \eqqcn{}{ \deg(X) = d_{Q^{op}}(u, t).}

\begin{lem}\label{nobackmap}
Assume that $X, Y$ are indecomposable in $\caM^\pm$ such that $\deg (X) < \deg(Y)$, then $\hhom_{\caM}(X, Y) =0$.
\end{lem}
\begin{proof}
The case for $\caM^+$ follows from (e) of Lemma 2.4 of \cite{i11}, and the case for $\caM^-$ follows from the duality.
\end{proof}

As a corollary, we have the following proposition.

\begin{pro}\label{orlovM}  $\caM^+$ and $\caM^-$ are Orlov categories.
\end{pro}

For a subcategory $\caC$ of $\mmod \GG$, let $^{\perp_n} \caC$ be the full subcategory of $\mmod \GG$ with objects $X$ satisfying $\Ext_{\GG}^t(X,\caC)=0$ for $1 \le t < n$ and let $\caC^{\perp_n} $ be the full subcategory of $\mmod \GG$ with objects $X$ satisfying $\Ext_{\GG}^t(\caC, X)=0$ for $1 \le t < n$.
A subcategory  $\caC$ of $\mmod \GG$ is called {\em $n$-rigid} if $\caC \subset \caC^{\perp_n} \cap ^{\perp_n}\caC$.

Clearly, both $\caM^+_{(u_l,0)}= \add D\GG$ and $\caM^-_{(u_1,0)}= \add \GG$ are $n$-rigid and are isomorphic to  $\bcaG_0 = \bcaG^+_{(u_l,0)}= \bcaG^-_{(u_1,0)} =\add \GG$, which is isomorphic to the quotient category $\caP(Q^{op,\perp})$ of bound path category defined  by the quiver $Q^{op}$ modulo the ideal generated by the relation $\rho^{op,\perp}$.

\medskip

Now we consider $\caM^-$, the case for $\caM^+$ follows from the duality.

\medskip

We have clearly that $\Theta^-_{(u_1,0)}$ sending $ \GG e_u \to \olG e_u$ defines an isomorphism from $\caM^-_{(u_1,0)}$ to $\bcaG^-_{(u_1,0)}$, with obvious inverse $\Psi^-_{(u_1,0)}$ sending $ \olG e_u \to \GG e_u$.

\medskip

Consider the following conditions for $(u,m)$:

\begin{cdns}\label{conditionaa}
\begin{enumerate}
\item\label{uma}$\Theta^-_{(u,m)}$ is an equivalence from $\caM^-_{(u,m)}$ to $\bcaG^-_{(u,m)}$ sending $ \ttn{-t}\GG e_{h}$  to $\olG e_{\tau^{-t} h}$.
    We have \eqqcn{isoGGm}{\ttn{-r}\GG e_h \simeq e\olG e_{\tau^{-r} h}} as $\GG$-modules and for $\ttn{-r}\GG e_h, \ttn{-r'}\GG e_{h'} $ in $\caM^-_{(u,m)} $, \eqqcn{isohom}{\hhom_{\caM^-_{(u,m)}} (\ttn{-r}\GG e_h, \ttn{-r'}\GG e_{h'}) \simeq \hhom_{\bcaG^-_{(u,m)}}(\olG e_{\tau^{-r}h}, \olG e_{\tau^{-r'} h'}).}

\item\label{umb}If $d_{Q^{op}}(u,m)\le d_{Q^{op}}(h,r) < d_{Q^{op}}(h,r-1) \le d_{Q^{op}}(u_l,0)$ we have an $n$-almost split sequence \small\eqqcn{nnasseq}{\arr{rl}{0 \slrw{} \ttn{-r+1}\GG¡¡e_{h} \slrw{} \bigoplus\limits_{(\tau^{-r'}h', n) \in H_{\tau^{-r}h}}  ( \ttn{-r'}\GG e_{h'} )^{\mu^{\tau_{-r}h} (\tau^{-r'}h', n)} & \slrw{} \cdots \\ \slrw{}  \bigoplus\limits_{(\tau^{-r'}h', 1) \in H_{\tau^{-r}h}}  ( \ttn{-r'}\GG e_{h'} )^{\mu_{\tau^{-r}h} (\tau^{-r'}h', 1)} & \slrw{} \ttn{-r}\GG¡¡e_{h} \lrw 0.}}\normalsize
\end{enumerate}

\end{cdns}

Clearly, Condition \ref{conditionaa} holds for $(u_1,0)$.
We first prove the following lemma.

\begin{lem}\label{inductivestep} Assume that $Q^-(u,m) $ is $\tau$-mature and $\caM^-(u,m)$ is $n$-rigid.
If $d_{Q^{op}}(u_1,0)\ge d_{Q^{op}}(u',m') = d_{Q^{op}}(u,m)+1$ and Condition \ref{conditionaa} holds for $(u',m')$, then Condition \ref{conditionaa} holds for $(u,m)$.
\end{lem}
\begin{proof}
Note that $\ind \caM^-_{(u,m)}$ is obtained by adding $\tau^{-m} \GG e_u$ to  $\ind \caM^-_{(u',m')}$.

Since $Q^{op}$ is a complete $\tau$-slice of $\olQ$, by Proposition \ref{KoszulinG}, we have a Koszul complex \eqqc{nnassinG}{\olG e_{\tau^{-m+1} u}\lrw \cdots \lrw Y_t \lrw \cdots \lrw Y_1 \lrw \olG e_{\tau^{-m} u},} with $Y_t \simeq \bigoplus_{(j,t)\in H^{\tau^{-m+1} u}}  (\olG e_h)^{\mu^{\tau^{-m+1} u}(j,t)}$.
This is an $n$-almost split sequence in $\bcaG^-_{(u,m)}$ since  $Q^-(u,m) $ is $\tau$-mature.
Applying the quasi inverse $\Psi^-_{(u',m')}$ of $\Theta^-_{(u',m')}$ on \eqqc{nnsourceinG}{\olG e_{\tau^{-m+1} u}\lrw \cdots \lrw Y_t \lrw \cdots \lrw Y_1,} the unique source sequence starting at $\olG e_{\tau^{-m+1} u}$ in $\bcaG_{(u',m')}$, we get a source sequence \eqqc{nnsource}{  \ttn{-m+1} \GG e_{u} \slrw{g_{n+1}} \cdots \slrw{g_{t+1}} X_t \lrw \cdots \slrw{g_1} X_1} in $\caM^-_{(u',m')} $ with \eqqcn{Xt}{ X_t \simeq \bigoplus_{(j,t)\in H^{\tau^{-m+1} u}}  (\ttn{-v(j)}\GG e_{u(j)} )^{\mu^{\tau^{-m+1} u}(j,t)},} here we have $j = (u(j), v(j))$ as a vertex in $ \zzs{n-1}Q^{op}$.

Since (\ref{uma}) of Condition \ref{conditionaa} holds for $(u',m')$, we have that \eqqcn{psinpmp}{\Psi_{(u',m')}(\olG e) = \GG \mbox{ and }\Psi_{(u',m')}(\olG e_{\tau^{-h} v}) = \ttn{-h} \GG e_v} for $d_{Q^{op}}(u_1,0) \ge d_{Q^{op}}(v,h)> d_{Q^{op}}(u,m)$.
Since \eqqcn{}{\arr{ccl}{\ttn{-h} \GG e_v &\simeq &\Hom_{\GG}(\GG ,\ttn{-h} \GG e_{v}) =\caM^-_{(u',m')} (\GG ,\ttn{-h} \GG e_{v}) \\ &\simeq &\caM^-_{(u',m')}\Hom_{\olG}(\Psi_{(u',m')}\olG e, \Psi_{(u',m')}\olG e_{\tau^{-h} v}) \\ & \simeq & \Hom_{\olG}(\olG e,\olG e_{\tau^{-h} v})  \simeq e \olG e_{\tau^{-h} v},} } we have \eqqcn{Yt}{eY_t \simeq X_t.}
By applying $\Hom_{\olG}(\olG e,\mbox{--})$ on \eqref{nnassinG}, we get an exact sequence of $\GG$-modules from \eqref{nnsource}, \eqqc{xzero}{0\lrw \ttn{-m+1} \GG e_{u} \slrw{g_{n+1}} \cdots \lrw X_t \lrw \cdots \slrw{g_1} X_1 \lrw e\olG e_{\tau^{-m}u} \lrw 0,} since \eqref{nnassinG} is an $n$-almost split sequence in $\bcaG_{(u,m)}$.

Note that \eqqcn{endo}{\End_{\GG} e\olG e_{\tau^{-m}u} \simeq e_{\tau^{-m}u}\olG e_{\tau^{-m}u} \simeq k,} so $ e\olG e_{\tau^{-m}u}$ is an indecomposable $\GG$-module.

Since $\caM^-_{(u,m)}$ is $n$-rigid, $\caM^-_{(u',m')}$ is $n$-rigid.
Let $Z_t$ be the cokernel of $g_{t+1}$ in \eqref{nnsource}, then \eqqcn{Zcoker}{Z_0\simeq e\olG e_{\tau^{-m}u} \mbox{ and }Z_{n+1} =\ttn{-m+1} \GG e_{u} .}
By the dual of 2.2.1 of \cite{i007} \eqqcn{}{\Ext^t_{\GG}(Z_0, \ttn{-r} \GG e_{v}) =0} for  $ \ttn{-r} \GG e_{v}$ in $\caM^-_{(u,m)}$ and $0< t < n$.
Thus $Z_0$ is in $^{\perp_n}\caM^-_{(u,m)}$.
If $r>0$, then $\ttn{-r} \GG e_{v}$ is not projective.
By Theorem 1.5 of \cite{i007}, we have that for $\ttn{-r} \GG e_{h}$ in $\caM^-_{(u,m)}$, \eqqcn{}{\Ext^t_{\GG}(\ttn{-r} \GG e_{v}, Z_0)=\Ext^{n-t}_{\GG}(Z_0, \ttn{-r+1} \GG e_{v}) =0 ,} for $0< t < n$.
This shows that $Z_0$ is in ${\caM^-_{(u,m)} }^{\perp_n}$, so $Z_0$ is in ${\caM^-_{(u,m)}}^{\perp_n}\cap {^{}}^{\perp_n}\caM^-_{(u,m)} $.
Let $ \caM' = \add(\{Z_0\} \cup \caM^-_{(u,m)})$, following the argument in the proof of Lemma  3.2 in \cite{i007}, we have the following exact sequences,
\eqqc{homcv}{\arr{l}{0\lrw \caM' (\mbox{--}, \ttn{-m+1} \GG e_{u}) \slrw{} \cdots \lrw \caM' (\mbox{--}, X_t) \lrw \cdots \\ \qquad \slrw{} \caM'(\mbox{--}, X_1 )\lrw \caM'(\mbox{--}, Z_0) \lrw \Ext_{}^1(\mbox{-- }, Z_1 ) \lrw 0,}}
and
\eqqc{homct}{\arr{l}{0\lrw \caM' (Z_0,\mbox{--}) \slrw{} \cdots \lrw \caM' ( X_t,\mbox{--}) \lrw \cdots \\ \qquad \slrw{} \caM'( X_n,\mbox{--} )\lrw \caM'(\ttn{-m+1} \GG e_{u},\mbox{--} ) \lrw \Ext_{}^1(Z_{n} , \mbox{--}) \lrw 0.}}
By 2.2.1(1) of \cite{i007}, we have \eqqcn{}{ \Ext_{\GG}^1 (\mbox{--}, Z_1 ) = \Ext_{\GG}^{n-1} (\mbox{--}, Z_n),} and \eqqc{embedF}{0 \lrw \Ext_{\GG}^{n-1} (\mbox{--}, Z_{n-1})\lrw \Ext^n_{\GG} (\mbox{--}, Z_n)\lrw \Ext_{\GG}^n (\mbox{--}, X_n)} is exact on $\caM^-_{(u,m)}$.
Note that $\caM^-_{(u,m)} \subset {}^{\perp_n} \GG$, so by Theorem 1.5 of \cite{i007}, we have that \eqqcn{}{ D\Ext^n_{\GG} (\mbox{--}, Z) \simeq \uhom(Z , \ttn{}\mbox{--}).}
Apply $D=\hhom_k(\mbox{--},k)$ on \eqref{embedF}, using this we get an exact sequence \eqqcn{FandGG}{\overline{\caM'}(X_n,\ttn{}\mbox{--})\lrw \overline{\caM'} (Z_n,\ttn{}\mbox{--}) \lrw D\Ext_{\GG}^{n-1} (\mbox{--}, Z_{n-1}) \lrw 0 .}
So we get exact sequence
\eqqc{FandG}{{\caM'}(X_n,\mbox{--})\lrw {\caM'} (Z_n,) \lrw D\Ext_{\GG}^{n-1} (\ttn{-1}\mbox{--}, Z_{n-1}) \lrw 0 }
on $\caM^-_{(u,m)}$.
Note that both $\bcaG^-_{(u,m)}$ and $\caM^-_{(u,m)}$ are Orlov categories and $\bcaG^-_{(u',m')}$ and $\caM^-_{(u',m')}$ are equivalent.
Since $\bcaG^-_{(u,m)}$ is $\tau$-mature, \eqref{nnsourceinG} is an $n$-almost split sequence in $\bcaG^-_{(u,m)}$, thus it is a source sequence in $\bcaG^-_{(u',m')}$, so \eqref{xzero} is also a source sequence in $\caM^-_{(u',m')}$.
Hence it is also a source sequence in $\caM^-_{(u,m)}$ by Lemma \ref{nobackmap}.
Using the argument of the proof of Proposition 3.3 of \cite{i007}, we get that it is a sink sequence in $\caM^-_{(u,m)}$, and $Z_0 \simeq \ttn{-1} Z_{n+1}$, that is \eqqcn{}{e\olG \tau^{-m}u \simeq \ttn{-m} \GG e_u} as $\GG$-modules.
So \eqref{xzero} is an $n$-almost split sequence in $\caM^-_{(u,m)}$ with $e\olG \tau^{-m}u \simeq \ttn{-m} \GG e_u$.
This proves (\ref{umb}) of Condition \ref{conditionaa} for $(u,m)$.

Extend $\Theta_{(u',m')}^-$ to $\Theta_{(u,m)}^-$ on $\caM^-_{(u,m)}$ by setting \eqqcn{}{\Theta_{(u,m)}^- (\ttn{-m}\GG e_{u}) = \olG e_{\tau^{-m}u}\\ = \olG e_{(u,m)},} we have $\Theta_{(u,m)}^- (\ttn{-r}\GG e_{v}) = \Theta_{(u',m')}^- (\ttn{-r}\GG e_{v})= \olG e_{(v,r)}$ for $\ttn{-r}\GG e_{v}$ in $\caM^-_{(u',m')}$.
Thus (\ref{umb}) of Condition \ref{conditionaa} for $(u,m)$ tells us that if \eqref{nnassinG} is a sink sequence,  so is \eqref{xzero}, since it is exactly the image of \eqref{nnassinG} under $\Theta_{(u,m)}^-$.
The quiver ${Q_{(u,m)}}$ is $\tau$-mature with $Q^{op}=Q_{(u_l,0)}$ an initial full subquiver which is a complete $\tau$-slice, and $\Theta_{(u,m)}^-$ restricts to the equivalence $\Theta_{(u_l,0)}^-$ on $\caM^-_{(u_l,0)}$.
So by  Lemma \ref{nobackmap} and Proposition \ref{relation:local},  $\Theta_{(u,m)}^-$ is an equivalence from $ \caM^-_{(u,m)}$ to $ \bcaG^-_{(u,m)}$.
Thus
\eqqcn{}{\Hom_{\olG}(\olG e_{(v,t)},\olG e_{(u,m)}) \simeq \caP(Q_{(u,m)})(\tau^{-t} v, \tau^{-m} u) \simeq \Hom_{\caM^-_{(u,m)}} (e\olG e_{\tau^{-t}v}, e\olG e_{\tau^{-m}u}) .}
This proves (\ref{uma}) of Condition \ref{conditionaa} for $(u,m)$.
Thus Condition \ref{conditionaa} holds for $(u,m)$.
\end{proof}

Write $Q(\caM^{-})$ for the full subquiver of $\zZ|_{n-1} Q^{op}$ with vertex set  \eqqcn{}{Q_0(\caM^{-}) = \{(j,r) \in \olQ_0| \ttn{-r} \GG e_j  \neq 0\}} and $Q(\caM^{+})$ for the full subquiver of $\zZ|_{n-1} Q^{op}$ with vertex set  \eqqcn{}{Q_0(\caM^{+})= \{(j,-r) \in \olQ_0| \ttn{r} De_j\GG  \neq 0\}.}
Let \eqqcn{}{\caG(\caM^-) =\add\{\olG e_j | j \in Q_0(\caM^{-}) \} \mbox{ and }\caG(\caM^+) =\add\{\olG e_j | j \in Q_0(\caM^{-}) \}.}
We have the following result.

\begin{thm}\label{equiv}
Assume that $\GG$ is an $n$-slice algebra as above.

If $Q(\caM^{-})$ is $\tau$-mature and $\caM^-$ is $n$-rigid, then $\caM^-$ have $n$-almost split sequence for each of its non-projective indecomposable object, and there is an equivalence $\Theta^-$  from $\caM^-$ to $\caG(\caM^-)$  such that \eqqcn{}{\tau^- \Theta^- \simeq \Theta^-\ttn{-}.}

If $Q(\caM^{+})$ is $\tau$-mature and $\caM^+$ is $n$-rigid, then $\caM^+$ have $n$-almost split sequence for each of its non-injective indecomposable object, and there is an equivalence  $\Theta^+$ from $\caM^+$ to $\caG(\caM^+)$ such that \eqqcn{}{\tau \Theta^+ \simeq \Theta^+\ttn{}.}
\end{thm}
\begin{proof}
We prove the first assertion, the second follows from duality.

It is obvious that $$\caM^- = \bigcup\limits_{\arr{c}{(u,m)\in Q_0(\caM^{-})\\ d_{Q^{op}}(u,m)\le d_{Q^{op}}(u_l,0)}} \caM^-_{(u,m)}$$ and $$\bcaG^- = \bigcup\limits_{\arr{c}{(u,m)\in Q_0(\caM^{-})\\ d_{Q^{op}}(u,m)\le d_{Q^{op}}(u_l,0) } } \bcaG^-_{(u,m)}.$$
And the first assertion holds for $\caM^-_{(u_1,0)} $ and $\caG(\caM^-_{(u_l,0)}) = \bcaG^-_{(u_1,0)}$

Since $Q(\caM^{-})$ is $\tau$-mature and $\caM^-$ is $n$-rigid, so each $Q^-_{(u,m)}$ is $\tau$-mature and each $\caM^-_{(u,m)}$ is $n$-rigid.
Take Lemma \ref{inductivestep} as the induction step, we immediately get $\caM^-$ have $n$-almost split sequence for each of its non-projective indecomposable object, and for $X$ in $ \caM^-$, define $ \Theta^- (X) = \Theta^-_{(u,m)} (X)$ if $X$ is in $\caM^-{(u,m)}$, then  $\Theta^-$ is an equivalence, by Proposition \ref{relation:local}, and we have \eqqcn{}{\tau^- \Theta^- \simeq \Theta^-\ttn{-}} from the definition of $\Theta^-$.
This proves the first assertion.
\end{proof}

We say $\caM^-$ (respectively, $\caM^+$) is {\em $\tau$-mature} if $Q(\caM^-)$ (respectively, $Q(\caM^+)$) is $\tau$-mature, and we say $\GG$ is {\em $\tau$-mature} if both $\caM^-$ and $\caM^+$ are $\tau$-mature.
As corollary of the above theorem, we have the following theorem about the Auslander-Reiten quiver of the $n$-preinjective and $n$-preprojective components of an $n$-slice algebra $\GG$.
\begin{thm}\label{dual:ARquiver}
Let $\GG$ be an $n$-slice algebra as above.

If $\caM^-$ is both $n$-rigid and $\tau$-mature then the Auslander-Reiten quiver of $\caM^-(\GG)$  is a truncation of the bound quiver $\zZ|_{n-1} Q^{op,\perp}$.

If $\caM^+$ is $n$-rigid and $\tau$-mature then the Auslander-Reiten quiver of $\caM^+(\GG)$  is a truncation of the bound quiver $\zZ|_{n-1} Q^{op,\perp}$.
\end{thm}

Now we consider when an algebra have $n$-rigid $\caM^-(\GG)$ and $\caM^+(\GG)$.
Since an $n$-hereditary algebra is either $n$-representation-finite or $n$-representation-infinite, the following lemma follows from Lemma 2.7 when $\GG$ is $n$-representation-finite, and from Corollary 4.7 when $\GG$ is $n$-representation-infinite.

\begin{lem}\label{nrigidM} If $\GG$ is an $n$-hereditary algebra,  then both $\caM^+$ and $\caM^-$ are $n$-rigid.
\end{lem}

By Lemma \ref{nrigidM}, we get immediately the following results from Theorem \ref{equiv}, since an $n$-slice algebra is $n$-hereditary.
\begin{thm}\label{exitsssaseqnn}Assume that $\GG$ is an $n$-slice algebra.

If $\caM^{-}$ is $\tau$-mature, then $\caM^-$ have $n$-almost split sequence for each of its non-projective indecomposable object, and there is an equivalence $\Theta^-$  from $\caM^-$ to $\caG(\caM^-)$  such that \eqqcn{}{\tau^- \Theta^- \simeq \Theta^-\ttn{-}.}

If $\caM^{+}$ is $\tau$-mature, then $\caM^+$ have $n$-almost split sequence for each of its non-injective indecomposable object, and there is an equivalence  $\Theta^+$ from $\caM^+$ to $\caG(\caM^+)$ such that \eqqcn{}{\tau \Theta^+ \simeq \Theta^+\ttn{}.}
\end{thm}

\begin{thm}\label{dual:ARquivernn}
Assume that $\GG$ is an $n$-slice algebra.
Then the Auslander-Reiten quiver of $\caM^+(\GG)$ and $\caM^-(\GG)$ are truncations of the quiver $\zZ|_{n-1} Q^{op,\perp}$ if they are $\tau$-mature.
\end{thm}

{\bf Remark.} Let $\LL= \GG^{!,op}$ and $\dtL$ be the trivial extension of $\LL$.
$\dtL$ is $n$-translation algebra, so it is $(n+1,q)$-Koszul for $q>0$, or $q=\infty$.
If $q=\infty$, then we have that $\caM^+(\GG)$ and $\caM^-(\GG)$ are always $\tau$-mature.
If $q$ finite, it is interesting to know when $\caM^+(\GG)$ and $\caM^-(\GG)$ are $\tau$-mature?

\subsection{Example.}
\label{sub:eg.hered}
{\em Preprojective and preinjective components of a path algebra.}

Let $Q$ be an acyclic quiver, and let $\GG=kQ$ be the path algebra defined on $Q$, then the relation set for $\GG$ is the empty set.
It is well known that the Auslander-Reiten quiver of the preprojective and preinjective components of finitely generated $kQ$-modules are truncations of the quiver $\zZ Q^{op}$ \cite{ars95,r84}.
This can be recovered using above results.
Let $\LL = kQ/(Q_2)$, where $Q_2$ is the set of paths of length $2$, $\LL$ is the quotient algebra of $kQ$ by the square of its radical.
Then $\LL = \GG^{!,op}$ is the quadratic dual of $\GG$, with vanishing radical square.

Now $\dtL $ is a symmetric algebra with vanishing radical cube. The returning arrow quiver $\tQ$ is exactly the double quiver of $Q$, obtained from $Q$ by adding a reversed arrow $\xa^*$ to each arrow $\xa$, with relation set \eqqcn{}{\arr{ll}{\trho = & \{\xa{\xa'} |\xa,{\xa'}\in Q_1, t({\xa'}) =s(\xa)\} \\ & \cup \{\xa^*\xa - {\xa'}^*{\xa'}|\xa,{\xa'}\in Q_1, s(\xa)=s({\xa'})\} \\ & \cup \{\xa^*\xa - {\xa'}^*{\xa'}|\xa,{\xa'}\in Q_1, s(\xa)=t({\xa'})\} \\ & \cup  \{\xa\xa^* -{\xa'}{\xa'}^*| \xa,{\xa'}\in Q_1,t(\xa)=t({\xa'})\}}} as the bound quiver of some twisted trivial extension $\dtnL$ of $\LL$.

Note that $\dtnL$ is $(2,\infty)$-Koszul when $\GG$ is representation infinite, and is $(2,h-2)$-Koszul when $\GG$ is representation finite, where $h$ is the Coxeter number for the  Dynkin quiver $Q$, by Corollary 4.3 of  \cite{bbk02}).
So $\dtnL$ is a stable $1$-translation algebra and $\GG$ is a $1$-slice algebra.

Consider now the algebra $\olL = \dtnL \# k\zZ^*$, its  bound quiver has quiver  $\mathbb Z|_0 Q =\mathbb ZQ $.
The relation set is \eqqcn{relation}{\arr{ll}{\olrho = & \{(\xa,r)({\xa'},r+1) | \xa,{\xa'}\in Q_1,t({\xa'})=s(\xa), r\in\zZ\} \\ & \cup \{(\xa^*,r)(\xa,r+1) - ({\xa'}^*,r)({\xa'},r+1)|\xa,{\xa'}\in Q_1, s(\xa)=s({\xa'}),r\in\zZ\}\\ & \cup \{(\xa^*,r)(\xa,r+1)- ({\xa'}^*,r)({\xa'},r+1)| \xa,{\xa'}\in Q_1, s(\xa)=t({\xa'}),r\in \zZ\} \\ & \cup \{(\xa,r)(\xa^*,r+1)-({\xa'},r)({\xa'}^*,r+1)| \xa,{\xa'}\in Q_1,t(\xa)=t({\xa'}),r\in\zZ\}.}}
$\zZ |_0 Q$ is a $1$-translation quiver with the $1$-translation $\tau (i,r)=(i,r-1)$.
The bound quiver $Q$, with all paths of length $2$ as relations, is a complete $\tau$-slice of the bound quiver $\zZ Q$ defined above.

The bound quiver $\zZ Q^{\perp}$ for the quadratic dual $\olG = \olL^{!,op}$ has the quiver $\zZ Q$ with relation set \eqqc{genmesh}{\olrho^{\perp} = \{\sum_{s(\xa)=i}(\xa^*,r)(\xa,r+1)+\sum_{t(\xa)=i}(\xa,r)(\xa^*,r+1) |i\in Q_0, r\in \zZ\}.}
This is the classical translation quiver $\mathbb ZQ$ with the mesh relations.

Clearly, $\caM^+$ and $\caM^-$ are both $1$-rigid.

If $\GG$ is representation infinite, we have that $Q$ is not Dynkin.
$\caM^+$ is the category of preinjective $\GG$-modules and $\caM^-$ is the category of preprojective $\GG$-modules.
Since $\dtL$ is Koszul, $\mathbb N^+  Q^{op}$ and $\mathbb N^-  Q^{op}
$ are $\tau$-mature subquivers in $\zZ Q^{op}$.
$Q(\caM^+) = \mathbb N^+  Q^{op}$ and $Q(\caM^-) = \mathbb N^-  Q^{op}$ are respectively their Auslander-Reiten quivers with the mesh relations \eqref{genmesh}.

If $\GG$ is representation finite, we have that $Q$ is Dynkin, and $\caM^+=\caM^- = \mmod \GG$ is the category of finitely generated $\GG$-modules.
In this case, the quadratic dual $\olG$ of $\olL$ is an $(h-1)$-translation algebra with $(h-1)$-translation $\tau_{\perp}$, and with $(h-1)$-translation quiver $\zZ Q^{op,\perp}$ as its bound quiver.
$Q(\caM^-)$ is the subquiver of $\zZ Q^{op,\perp}$ obtained by taking the union of the  $\tau_{op,\perp}$-hammocks of the vertices of $Q$ in $\zZ Q^{op,\perp}$.
In this case, for each $v\in Q_0$, if $\tau v $ is in $Q(\caM^-)$, then $ \tau_{\perp }^{-1} v$ is not a vertex of $Q(\caM^-)$.
Thus  $Q(\caM^-)$ is a  $\tau$-mature subquiver of $\zZ Q^{^{op},\perp}$ with the restriction of  \eqref{genmesh} as  the mesh relations on $Q(\caM^-)$.
In fact, $Q(\caM^-)$ is exactly the Auslander-Reiten quiver of $\GG$.

The above idea  for the finite type is used in \cite{gl16} to study the iterated construction of the $n$-cubic pyramid algebras, whose quadratic dual are $(n-1)$-Auslander $n$-representation-finite algebras of type $A$ (absolutely $ n$-complete algebras in \cite{i11}).
In fact,  \cite{gl16} can be regarded as an example of $\ttn{}$-finite case of our main result in this section, which tell us how to find $\tau$-mature subquiver of  $Z|_{n-1} Q^{op}$ for certain $(n-1)$-representation-finite algebra $kQ/(\rho)$ of the type $A$ with bound quiver $Q= (Q_0,Q_1,\rho)$.
So our construction can also be regarded as a generalization of the cone construction in \cite{i11}
\section{$\nu_n$-closure and   $\mathbb Z|_{n-1} Q$}
\label{sec:nu}
Let $\GG$ be an $n$-slice algebra as in the last section.
By Theorem \ref{equiv}, there are  functors  $\Theta^{+}: \caM^{+} \to {\bcaG}$ and   $\Theta^{-}: \caM^{-} \to \bcaG$ for the $n$-preprojective component  $\caM^{-}$ of $\GG$ and $n$-preinjective component $\caM^{+}$ of $D\GG$, respectively, in the module  category of $\GG$.
Now we consider a version of such functors in derived category.

Let $\GG$ be a finite-dimensional algebra of global dimension $\le n$, and let
\eqqc{Nakayama}{\arr{l}{\nu  = D \circ  \mathbb R \hhom_{\GG} (- ,\GG)\simeq  D\GG \otimes_{\GG}^{\mathbb L} -: \mathcal D^b(\GG) \to \mathcal D^b(\GG)\\
\nu^{-}  =   \mathbb R \hhom_{\GG} (D - ,\GG)\simeq \mathbb R \hhom_{\GG} (D\GG,-): \mathcal D^b(\GG) \to \mathcal D^b(\GG)\\
}}
be the Nakayama functors.
They are autoequivalences of $ \mathcal D^b(\GG)$ and  are quasi-inverse one another.
The Nakayama functor is Serre functor and satisfies the functorial isomorphism (see \cite{h88,bk90}) \eqqc{}{\hhom_{\Db(\GG)}(X,Y) \simeq D \hhom_{\Db(\GG)}(Y, \nu X).}

Let $\nu_n = \nu[-n]:\mathcal D^b(\GG) \to \mathcal D^b(\GG)$, it is an auto-equivalence of $\mathcal D^b(\GG)$.
Regard $\mmod \GG $ as a subcategory of $\Db(\GG)$ with objects concentrated in degree zero, then (see \cite{i11})
\eqqcn{}{\tau_n=\Ho^0(\nu_n-) \mbox{, } \tau_n^-=\Ho^0(\nu_n^{-1}-), } and \eqqcn{}{ \hhom_{\Db(\GG)}(X,Y[n]) \simeq D \hhom_{\Db(\GG)} (Y, \nu_n X).}
For each projective $\GG$-module $\GG e_i$, we have \eqqcn{}{\nu_n (\GG e_i) \simeq D(e_i \GG)[-n], \mbox{ and } \nu_n^- D(e_i\GG)  = \GG e_i[n].}

Let \eqqcn{}{\caU =\caU(\GG) = \add\{\nu_n^t \GG | n\in \zZ\},} and  call it the {\em $\nu_n$-closure} of $\GG$ in $\mathcal D^b(\GG)$.
Let \eqqcn{}{\arr{l}{\caU_0^- = \add\{\nu_n^t \GG e_j | j\in Q_0, t\le 0 \mbox{ and } \ttn{t} \GG e_j \neq 0\} \\ \caU_0^+ = \add\{\nu_n^t D\GG e_j[-n] | j\in Q_0, t\ge 0 \mbox{ and } \ttn{t-1} D e_j \GG \neq 0\} ,}}
and let \eqqcn{}{\caU_0= \add( \caU_0^-\cup \caU_0^+).}
If $\GG$ is not $\ttn{}$-finite, then $\caU = \caU_0$.

Now assume that $\GG$ is an acyclic $n$-slice algebra with bound quiver $Q^{\perp}=(Q_0,Q_1,\rho^{\perp})$ and let $\LL$ be its quadratic dual with bound quiver $Q=(Q_0,Q_1,\rho)$.
Then $Q$ is a complete $\tau$-slice of the acyclic $n$-translation quiver $\olQ=\zZ|_n Q$, and $\olL =\olQ/(\rho)$ is an $n$-translation algebra.
Let $\olG$ be the quadratic dual of $\olL$ with bound quiver $\zZ|_n Q^{\perp} =(\olQ_0,\olQ_1, \overline{\rho}^{\perp} )$ and let $\bcaG =\add(\olG)$.
We now prove that under certain conditions, the $\nu_n$-closure of $\GG$ in the derived category  $\mathcal D^b(\GG)$ is equivalent to $\bcaG$.

We first prove the following lemma.

\begin{lem}\label{orlovdev}
If $\GG$ is acyclic, then $\caU$ is an Orlov category.
\end{lem}
\begin{proof}
Using the path order $Q_0=\{u_1,\ldots,u_l\}$ of $Q$.
Define \eqqcn{}{\deg(\nu_n^{r} \GG e_{u_t})< \deg( \nu_n^{r'} \GG e_{u_{t'}})} if $r < r'$, or $r=r'$ and $u_t <u_{t'}$.
Then \eqqcn{}{\Hom_{\caU}(\nu_n^{r} \GG e_{u_t}, \nu_n^{r'} \GG e_{u_{t'}}) =0} for $r < r'$, by Proposition 2.3 of \cite{hio14}.
And  \eqqcn{}{\Hom_{\caU}(\nu_n^{r} \GG e_{u_t}, \nu_n^{r} \GG e_{u_{t'}}) = \Hom_{\caU}( \GG e_{u_t},  \GG e_{u_{t'}}) = \Hom_{\GG}( \GG e_{u_t},  \GG e_{u_{t'}})=0} for $t < t'$, since $\nu_n$ is an auto-equivalence.
So \eqqcn{}{\Hom_{\caU}(\nu_n^{r} \GG e_{u_t}, \nu_n^{r'} \GG e_{u_{t'}}) =0} if $\deg(\nu_n^{r} \GG e_{u_t})< \deg( \nu_n^{r'} \GG e_{u_{t'}})$, and $\caU$ is an Orlov category.
\end{proof}

By Lemma \ref{globaldim}, $\GG$ is always Koszul algebra.
If $\dtL$ is Koszul for its quadratic dual $\LL =\GG^{!,op}$, we say $\GG$ is {\em of infinite type}.
In this case, $Q(\caM^+)$ and $Q(\caM^-)$ are both $\tau$-mature.
By Theorem \ref{equiv}, we see that  $\caM^+$ and $\caM^-$ have $n$-almost split sequences if they are $n$-rigid.
Consider the following conditions:
\begin{cdns}\label{conditionbb}
\begin{enumerate}
\item \label{cbone}$\caM^- $ is $n$-rigid, and for $i,j\in Q_0$, $s,t \ge 0$, $$\hhom_{\caU_0^-} (\nu_n^{-s} \GG e_i,\nu_n^{-t} \GG e_j) \simeq \hhom_{\GG} (\ttn{-s} \GG e_i,\ttn{-t} \GG e_j)$$ when $\nu_n^{-s} \GG e_i,\nu_n^{-t} \GG e_j$ are in $\caU^-_0$.

\item \label{cbtwo}$\caM^+ $ is $n$-rigid, and  for $i,j\in Q_0$, $s,t \ge 0$,  $$\hhom_{\caU_0^+} (\nu_n^s \GG e_i,\nu_n^t \GG e_j) \simeq \hhom_{\GG} (\ttn{s-1}D \GG e_i,\ttn{t-1} D\GG e_j)$$ when $\nu_n^{s} \GG e_i,\nu_n^{t} \GG e_j$ are in $\caU^+_0$.
\end{enumerate}
\end{cdns}
We have the following Lemma.

\begin{lem}\label{ZQ}
Let $\GG$ be an acyclic $n$-slice algebra  of infinite type.

If $\GG$ satisfies (\ref{cbone}) of Condition \ref{conditionbb}, then $\caU$ has source sequences.

If $\GG$ satisfies (\ref{cbtwo}) of Condition \ref{conditionbb}, then $\caU$ has sink  sequences.
\end{lem}
\begin{proof}
We prove the first assertion, the second  follows from the duality.

Assume that $\GG$ satisfies (\ref{cbone}) of Condition \ref{conditionbb}.
For each indecomposable objects $X$ in $\caU$, there is an integer $r$ and an $i\in Q_0$ such $X =\nu_n^{r}  \GG e_i $.
By Theorem \ref{equiv}, we have an $n$-almost split sequence \eqqc{devnass}{M_{n+1}= e_{i}\GG \to M_n \to \cdots \to M_1 \to M_0 = \ttn{-1} e_i \GG }
in $\caM^-$.
By Lemma \ref{orlovdev} and the proof of Lemma \ref{inductivestep}, we have that for each $j\in Q_0$ and each $t=0,1,\ldots, n+1$, there are integers $0\le r(t,j) \le 1$ such that \eqqcn{}{M_t \simeq \bigoplus (\ttn{-r(t,j)}\GG e_j)^{s(t,j)}.}
Take \eqqcn{}{M'_t = \bigoplus (\nu_n^{-r(t,j)}\GG e_j)^{s(t,j)}} for $t=0,1,\ldots, n+1$.

Now let \eqqcn{}{N_t = \bigoplus (\nu_n^{-r(t,j)+r}\GG e_j)^{s(t,j)}} for $t=0,1,\ldots, n+1$, then $X=N_{n+1}$, and we have a sequence
\eqqc{sourcesequenceinU}{X= N_{n+1} = \nu_n^{r}\GG e_i \to N_n \to \cdots \to N_1 \to N_0= \nu_n^{r+1}\GG e_i}
in $\Db(\GG)$ with each term in $\caU$.

For each indecomposable objects $Z\in \caU$, \eqqcn{}{Z = \nu_n^h \GG e_{j'}} for some $h\in\zZ$, $j'\in Q_0$.
Let \eqqcn{}{Z' = \nu_n^{h-r} \GG e_{j'}.}

If $h >r$ then \eqqcn{}{hl > (-r(t,j)+r)l ,} and we have \eqqcn{}{\Hom_{\caU}(N_t, \nu_n^h \GG e_{j'}) =0} by Lemma \ref{orlovdev}.

If $h\le r$, then $h-r \le 0$ and  $Z'$ is in $\caU^-$.
So we have a commutative diagram with isomorphic vertical maps
\small\eqqc{devcoma}{
\arr{ccccccccc}{
\cdots \to &\hhom_{\caU}(N_{n-1},Z )& \to &\hhom_{\caU}(N_n,Z) &\to &\hhom_{\caU}(X,Z)\\
&\downarrow\simeq &&\downarrow\simeq &&\downarrow\simeq &\\
\cdots \to &\hhom_{\caU}(M'_{n-1},Z' ) & \to&\hhom_{\caU}(M'_n,Z') &\to &\hhom_{\caU}(M'_{n+1},Z')\\
}}
\normalsize
since $\nu_n$ is an autoequivalence on $\caU$.

Let \eqqcn{}{Z'' = \ttn{h-r} \GG e_{j'},} then $Z''$ is in $\caM^{-}$, and we have a commutative diagram with isomorphic vertical maps
\small\eqqc{devcomb}{
\arr{cccccccccccc}{
\cdots \to &\hhom_{\caU}(M'_{n-1},Z' ) & \to&\hhom_{\caU}(M'_n,Z') &\to &\hhom_{\caU}(M'_{n+1},Z') &\\
&\downarrow\simeq &&\downarrow\simeq &&\downarrow\simeq &\\
\cdots \to&\hhom_{\GG}(M_{n-1},Z'' ) & \to &\hhom_{\GG}(M_n,Z'') &\to  &\hhom_{\GG}(M_{n+1},Z'') &
 \\
}}\normalsize
by our assumption.
Since \eqref{devnass} is an $n$-almost split sequence, we have an exact sequence
\eqqcn{devsource}{\cdots \to \hhom_{\GG}(M_{n-1},Z'' )  \to \hhom_{\GG}(M_n,Z'') \to  \hhom_{\GG}(M_{n+1},Z'') \to  \caS_{M_{n+1}} \to 0,}
where $\caS_{M_{n+1}} $ is the simple $\caM$-module with \eqqcn{}{\caS_{M_{n+1}}(Z'') =0 } for indecomposable $Z''\not\simeq M_{n+1}$ and  \eqqcn{}{\caS_{M_{n+1}}(M_{n+1}) =k =\hhom_{\GG}(M_{n+1},M_{n+1}).}
Thus \eqqcn{}{\hhom_{\GG}(M_{n+1},Z'') = J_{\GG}(M_{n+1},Z'')} when $M_{n+1} \not\simeq Z''$, and \eqqcn{}{ J_{\GG}(M_{n+1},M_{n+1})=0.}
So \eqqcn{devssforMn}{\cdots \to \hhom_{\GG}(M_{n-1},Z'' )  \to \hhom_{\GG}(M_n,Z'') \to  J_{\GG}(M_{n+1},Z'') \to  0} for any indecomposable $Z''$ in $\caM^{-}$.
Using \eqref{devcoma} and \eqref{devcomb}, we get that \eqqcn{devsource}{\cdots \to \hhom_{\caU}(N_{n-1},Z )  \to \hhom_{\caU}(N_n,Z) \to  J_{\caU}(N_{n+1},Z) \to  0,} for any indecomposable $Z$ in $\caU$.
This proves that \eqref{sourcesequenceinU} is a source sequence in $\caU$, thus $\caU$ has source sequences.
\end{proof}

\begin{thm}\label{dual:derequiv}
Assume that  $\GG$ is an acyclic $n$-slice algebra  of infinite type.
If $\GG$ satisfies one of the conditions in Condition \ref{conditionbb}, then there is an equivalence $\Theta$  from $\caU$ to $\bcaG$ such that \eqqcn{}{\tau \Theta \simeq \Theta\nu_n.}
\end{thm}
\begin{proof} Assume that $\GG$ is an $n$-slice algebra  of infinite type, then it is an $n$-representation-infinite algebra  and  satisfies (\ref{cbone}) of Condition \ref{conditionbb}, we have $\caU^- $ is equivalent to $\caM^-$ which assigns $\nu_n^{-t}\GG e_{u} $ to $ \ttn{-t} \GG e_{u}$.
Since $\GG$ is  of infinite type, $\caU =\caU_0$ and each convex full subquiver of $\zzs{n-1} Q^{op}$ is $\tau$-mature.

By Theorem \ref{equiv}, \eqqcn{}{\Theta^- : \ttn{-t} \GG e_{u} \to \olG e_{\tau^{-t}u} } defines an equivalent from $\caM^-$ to $\bcaG^-$.
So we have an equivalence from $\caU^- $ to $\bcaG^-$ sending $\nu_n^{-t}\GG e_{u} $ to $ \olG e_{\tau^{-t}u} $ for $t\ge 0$.
Clearly this extends to a correspondence $\Theta$ from the objects of $\caU$ to $\bcaG$ sending $\nu_n^{-t}\GG e_{u} $ to $ \olG e_{\tau^{-t}u} $ for all $t$.
Thus \eqqcn{}{Q_{\caU,0} =\{(u,t) | \nu^{-t} u\in \caU \} = \zZ Q^{op}_0} and \eqqcn{}{Q_{\caU} =\zzs{n-1}Q^{op},} $Q_{\caU^-}$ is a terminating full subquiver of $Q_{\caU}$ which obviously contains a complete $\tau$-slice.
By Lemma \ref{ZQ}, $\caU$ has source sequences, and by the proof of Lemma \ref{ZQ}, it satisfies (iv) of Proposition \ref{relation:local}.
So by Lemma \ref{orlovdev} and Proposition \ref{relation:local},  $\Theta$ is an equivalence from $\caU$ to $\bcaG$.
We clearly have that \eqqcn{}{\Theta(\nu_n^{t} \GG e_u) = \tau^{t}\Theta(\GG e_u)}  for $(u,t)\in \zzs{n-1}Q^{op}_0$, so we get the required equivalence.
\end{proof}

As a corollary of Theorem \ref{dual:derequiv}, we can describe the  Auslander-Reiten quiver of the $\nu_n$-closure, similar to Theorem \ref{dual:ARquiver}.

\begin{thm}\label{derived:ARquiver}
Let $\GG$ be  an acyclic $n$-slice algebra of infinite type with bound quiver $Q^\perp$.
If $\GG$ is satisfies one of the Condition  \ref{conditionbb}, then the Auslander-Reiten quiver of the $\nu_n$-closure $\caU(\GG)$ of $\GG$ in $\Db(\GG)$ is $\zZ|_{n-1} Q^{op,\perp}$.
\end{thm}

Recall that an algebra $\GG$ is called {\em $n$-representation infinite} if $\nu_n^{-t} \GG $ are modules for all $t \ge 0$ (\cite{hio14}).

\begin{cor}\label{taunclosure:nass}
If an $n$-slice algebra $\GG$ of infinite type, then

(1). $\caU$ has $n$-almost split sequences;

(2). $\caU$ is equivalent to $\bcaG$.

(3). the Auslander-Reiten quiver $\caU$ is $\zZ|_{n-1} Q^{op,\perp}$.
\end{cor}

\begin{proof}
We prove (1).
By Theorem 1.1 of \cite{hio14}, $\caM^+$ and $\caM^-$ are both $n$-rigid,.

Since $\GG$ is $n$-representation infinite, we have $\nu_n^{-t} \GG e_j$ and $\nu_n^{t}D e_j \GG $ are modules and thus  $$\arr{cl}{\nu_n^{-t} \GG e_j = H_0(\nu_n^{-t} \GG e_j) =\ttn{-t} \GG e_j & t\ge 0\\ \nu_n^{t}\GG =\nu_n^{t-1} D e_j\GG[-n] = H_0(\nu_n^{t-1}D e_j \GG ) =\ttn{t-1} D e_j \GG & t>0}$$ for all $j\in Q_0$ and $t\ge 0$.
This shows that $\GG$ satisfies both of the conditions in Condition  \ref{conditionbb}, and $\caU$ has $n$-almost split sequences by Lemma \ref{ZQ}.

(2) and (3) follows from Theorem \ref{dual:derequiv} and Theorem \ref{derived:ARquiver}.
\end{proof}

It is interesting to know when an $n$-slice algebra is $n$-representation-finite?
\subsection{Example.}
\label{sub:eg.dev}

If $Q$ is not Dynkin, the path algebra $\GG=kQ$ is representation-infinite.
It is known that $\GG$ is  of infinite type and satisfies conditions in Lemma \ref{ZQ}.
So our results recover the well known fact that $\zZ Q^{op} = \zZ Q$ appears in the Auslander-Reiten quiver of the derived category of $kQ$ as  components containing the shifts of preprojective and preinjective modules \cite{h88}.

{}
\end{document}